\documentclass[reqno]{amsart}
\usepackage{amsthm,amsmath,amssymb}
\usepackage{stmaryrd,mathrsfs}
\usepackage[T1]{fontenc}
\usepackage{esint}
\usepackage{color}

\allowdisplaybreaks

\newcommand{\ud}[0]{\,\mathrm{d}}

\newcommand{\dist}[0]{\operatorname{dist}}

\newcommand{\abs}[1]{|#1|}
\newcommand{\Babs}[1]{\Big|#1\Big|}

\newcommand{\Norm}[2]{\|#1\|_{#2}}

\newcommand{\BNorm}[2]{\Big\|#1\Big\|_{#2}}
\newcommand{\pair}[2]{\langle #1,#2 \rangle}
\newcommand{\Bpair}[2]{\Big\langle #1,#2 \Big\rangle}
\newcommand{\ave}[1]{\langle #1\rangle}

\newcommand{\bddlin}[0]{\mathscr{L}}

\newcommand{\BMO}[0]{\operatorname{BMO}}
\newcommand{\supp}[0]{\operatorname{supp}}
\newcommand{\loc}[0]{\operatorname{loc}}

\newcommand{\sign}[0]{\operatorname{sgn}}

\newcommand{\R}{\mathbb{R}}

\newcommand{\Z}{\mathbb{Z}}
\newcommand{\T}{\mathbb{T}}

\newcommand{\Exp}[0]{\mathbb{E}}

\newcommand{\eps}[0]{\varepsilon}

\def\cyr{\fontencoding{OT2}\fontfamily{wncyr}\selectfont}
\DeclareTextFontCommand{\textcyr}{\cyr}

\swapnumbers \numberwithin{equation}{section}

\theoremstyle{plain}
\newtheorem{theorem}[equation]{Theorem}
\newtheorem{proposition}[equation]{Proposition}
\newtheorem{corollary}[equation]{Corollary}
\newtheorem{lemma}[equation]{Lemma}

\theoremstyle{definition}
\newtheorem{definition}[equation]{Definition}

\theoremstyle{remark}
\newtheorem{remark}[equation]{Remark}

\makeatletter
\@namedef{subjclassname@2010}{%
  \textup{2010} Mathematics Subject Classification}
\makeatother

\begin{document}

\title{Pointwise convergence of vector-valued Fourier series}

\author[T. P. Hyt\"onen]{Tuomas P. Hyt\"onen}
\address{Department of Mathematics and Statistics, P.O.B.~68 (Gustaf H\"all\-str\"omin katu~2b), FI-00014 University of Helsinki, Finland}
\email{tuomas.hytonen@helsinki.fi}
\thanks{T.H. is supported by the European Union through the ERC Starting Grant ``Analytic--probabilistic methods for borderline singular integrals'', and by the Academy of Finland through projects 130166 and 133264.}

\author[M.T. Lacey]{Michael T. Lacey}
\address{School of Mathematics,
Georgia Institute of Technology, Atlanta GA 30332, United States}
\email{lacey@math.gatech.edu}
\thanks{M.L. is supported in part by the NSF grant 0968499, and a grant from the Simons Foundation (\#229596 to Michael Lacey). }

\date{\today}




\maketitle

\begin{abstract}
We prove a vector-valued version of Carleson's theorem: Let $Y=[X,H]_\theta$ be a complex interpolation space between a UMD space $X$ and a Hilbert space $H$. For $p\in(1,\infty)$ and $f\in L^p(\T;Y)$, the partial sums of the Fourier series of $f$ converge to $f$ pointwise almost everywhere. Apparently, all known examples of UMD spaces are of this intermediate form $Y=[X,H]_\theta$. In particular, we answer affirmatively a question of Rubio de Francia on the pointwise convergence of Fourier series of Schatten class valued functions.
\end{abstract}


\section{Introduction}

We are interested in the vector-valued extension of Carleson's celebrated theorem on pointwise convergence of Fourier series \cite{Carleson:66}, in the setting where the functions take values in an infinite-dimensional Banach space. 
Extending our recent model result for the Walsh series \cite{HL:Walsh},  we show that Carleson's theorem holds for a class of 
UMD (unconditionality of martingale differences) spaces that includes all known examples of such spaces.  

This class is referred to as \emph{intermediate UMD spaces:}
$Y$ is a complex interpolation space $Y=[X,H]_\theta$ between another UMD space $X$ and a Hilbert space $H$, where $\theta\in(0,1)$. 
This includes all UMD lattices \cite[Corollary on p.~216]{RdF:LNM}. It also includes the Schatten ideals $C_p$, $p\in(1,\infty)$, a principal example of non-lattice UMD spaces. It is an open question whether every UMD space is of this form.

Denote the Fourier transform and its partial sums of, say, a Schwartz function $f$ by
\begin{equation*}
	 \hat f (\xi ) := \int_{\R} f (x) e ^{-2 \pi i x \cdot \xi } \ud x,\qquad
	S_{m,n} f(x):= \int _{m } ^{n} \hat f (\xi ) e ^{2 \pi i x \cdot \xi }\ud\xi.
\end{equation*}
Then $S_{m,n}f(x)$ is also given by the convolution of $f$ with
\begin{equation*}
  x\mapsto e^{i\pi x(m+n)}\frac{\sin(\pi x(n-m))}{\pi x},
\end{equation*} 
which belongs to $L^q(\R)$ for all $q>1$. Thus, by H\"older's inequality, the map $f\mapsto S_{m,n}f(x)$ is well-defined from $L^p(\R;Y)$ to $Y$ for $p\in(1,\infty)$ and any Banach space $Y$. Our main result reads as follows:

\begin{theorem}\label{thm:main}
Let $Y$ be an intermediate UMD space, $p\in(1,\infty)$, and $f\in L^p( \mathbb R ;Y)$. Then
\begin{equation*}
	S_{m,n} f(x) \to f(x)
\end{equation*}
as $m\to-\infty,n\to\infty$ for a.e.~$x\in \R $.
In fact,  the  maximal partial sum operator $S^*$,
\begin{equation*}
  S^*f(x):=\sup_{\substack{m,n\in\R\\ m<n}}\abs{S_{m,n} f(x)},
\end{equation*}
is bounded from $L^p( \mathbb R ;Y)$ to $L^p( \mathbb R )$.
\end{theorem}

A necessary condition for the conclusions of Theorem~\ref{thm:main} is that $Y$ be a UMD space; see Proposition~\ref{prop:nec} for details. By transference, Theorem~\ref{thm:main} also implies an analogous statement for Fourier series of $f\in L^p(\T;Y)$; see Theorem~\ref{thm:T}.

Until now, this theorem was only known in UMD lattices, where it is due to Rubio de Francia \cite{RdF:LNM}. His proof relies on the original Carleson theorem, which is applied pointwise, using a representation of the lattice as a function space. A weaker statement on the behaviour of the partial sums was recently obtained by Parcet, Soria and Xu \cite{PSX} in general UMD spaces, where they proved that $S_{-n,n}f(x)=o(\log\log n)$ almost everywhere.

Rubio de Francia raised the question \cite[Problem 4 on p.~220]{RdF:LNM} whether Carleson's theorem extends to functions taking values in a Schatten ideal $C_p$ for $p\in(1,\infty)$. Since these spaces are special cases of intermediate UMD spaces, Theorem~\ref{thm:main} answers this in the affirmative. 

The proof is based on an elaboration of our method from \cite{HL:Walsh} where, adapting  the phase plane analysis \`a la Lacey--Thiele \cite{LT:MRL}, we treated the model case of Walsh--Fourier series. The key point of this strategy is that a critical quasi-orthogonality estimate, which we call \emph{tile-type}, can be obtained in the intermediate UMD spaces by interpolating between strong orthogonality estimates in a Hilbert space and a weaker version available in general UMD spaces.

The main difficulty in extending this strategy to the trigonometric Fourier setting is the lack, even in a Hilbert space, of a clean orthogonality estimate that would serve as an end-point of our interpolation. Because of this, we need to look more carefully at the phase-plane quasi-orthogonality estimates even in the classical $L^2$ setting, and we formulate a version of these bounds, which appears to be new. This leads to our notion of \emph{Fourier tile-type} of a Banach space, which is shown to be amenable to interpolation. We can then establish this property for all the intermediate UMD spaces as in Theorem~\ref{thm:main}. Once these critical quasi-orthogonality bounds are available, it is possible to once again adopt the general proof scheme of Lacey--Thiele \cite{LT:MRL}; we also frequently borrow from \cite{Thiele:book}. While the Fourier tile-type property is applied exactly once in the proof, UMD is used more frequently to control various Calder\'on--Zygmund type operators produced in the analysis.

\medskip

The plan of the paper is as follows. We first discuss the necessity of UMD for the main result, and show how to pass from the Fourier integrals on $\R$ to Fourier series on $\T$. This section is independent of the the rest of the paper. Turning to the proof of Theorem~\ref{thm:main}, we first recall some relevant notions from the phase plane analysis in two sections, and then turn to a thorough study of our new notion of Fourier tile-type in the next three sections. The last three sections of the paper then give a proof of Theorem~\ref{thm:main}, first in the somewhat simpler case of a large $p$ and finally, after developing some improved square-function estimates, in the case of a general $p>1$.

%

\section{Around the main theorem}

Theorem~\ref{thm:main} asserts the $L^p$ boundedness of the maximal partial sum operator $S^*$, thus \emph{a fortiori} the uniform  boundedness on $L^p(\R;Y)$ of the individual partial sum operators $S_{m,n}$. It is well-known that this is equivalent to the $L^p(\R;Y)$-boundedness of the Hilbert transform, and hence to the UMD property of $Y$. The following result addresses the necessity of the UMD property for the qualitative part of Theorem~\ref{thm:main}.

\begin{proposition}\label{prop:nec}
Let $Y$ be a Banach space, $p\in(1,\infty)$, and suppose that for every $f\in L^p(\R;Y)$, we have
\begin{equation*}
  S_{m,n}f(x)\to f(x)
\end{equation*}
as $m\to-\infty$, $n\to\infty$ for almost every $x\in\R$. Then $Y$ is a UMD space.
\end{proposition}

\begin{proof}
The proof resembles other similar arguments found in the literature, and we will be sketchy in some details. We have particularly borrowed from Torrea and Zhang \cite[proof of Theorem~D]{TZ}.

Consider an $x\in\R$ where the convergence takes place.
Then, for some $N$, we have $\abs{S_{m,n}f(x)}\leq\abs{f(x)}+1$ for all $m\leq -N,n\geq N$. On the other hand, by estimating the $L^{p'}$ norm of the convolution kernel $e^{i\pi x(n+m)}\sin(\pi x(n-m))/\pi x$ of $S_{m,n}$, we find that $\abs{S_{m,n}f(x)}\leq C_N\Norm{f}{p'}$ for $\abs{m},\abs{n}\leq N$. Thus $\abs{S_{0,n}f(x)}\leq C_N\Norm{f}{p'}$ if $n\in[0,N]$, and
\begin{equation*}
  \abs{S_{0,n}f(x)}
  =\abs{S_{-N,n}f(x)-S_{-N,0}f(x)}
  \leq(\abs{f(x)}+1)+C_N\Norm{f}{p'}
\end{equation*}
for $n>N$. Hence, at almost every $x\in\R$,
\begin{equation*}
  \sup_{n>0}\abs{S_{0,n}f(x)}<\infty.
\end{equation*}
We then pick an auxiliary $0\leq\phi\in C^1_c(\R_+)$ of integral one. By averaging, it follows that also
\begin{equation*}
  T^*f(x):=\sup_{r>0}\abs{T_r f(x)}:=\sup_{r>0}\Babs{\int_0^\infty S_{0,n}f(x)\phi(\frac{n}{r})\frac{\ud n}{r}}<\infty.
\end{equation*}

Let $L^0(\R)$ be the space of all measurable functions of $\R$, equipped with the topology of local convergence in measure: $f_j\to 0$ in $L^0(\R)$ if and only if $\abs{E\cap\{\abs{f_j}>\lambda\}}\to 0$ for all measurable sets $E$ with $\abs{E}<\infty$ and all $\lambda>0$. Clearly `all measurable sets $E$' can be replaced by all balls $B(0,k)$, $k=1,2,3,\ldots$. Then $L^0(\R)$ is metrizable, for example with
\begin{equation*}
  d(f,g):=\sum_{k=1}^\infty 2^{-k}\int_{B(0,k)}\min\{1,\abs{f-g}\}\ud x.
\end{equation*}
By above, we have that $T^*$ maps $L^p(\R;Y)$ into $L^0(\R)$, and by the closed graph theorem, we can deduce that $T^*$ is bounded between these spaces.

The operator $T_r$ has convolution kernel
\begin{equation*}
   K_r(x)=\int_{\R_+} e^{i\pi nx}\frac{\sin(\pi nx)}{\pi x}\phi(\frac{n}{r})\frac{\ud n}{r}
   =\int_{\sign(x)\R_+}e^{i\pi ur}\frac{\sin(\pi ur)}{\pi x}\phi(\frac{u}{x})\frac{\ud u}{x}.
\end{equation*}
Using the uniform boundedness of $e^{i\pi ur}\sin(\pi ur)$ and differentiating in $x$ under the integral, it is easy to check the bounds
\begin{equation*}
  \abs{K_r(x)}\leq\frac{C}{\abs{x}},\qquad\abs{K_r'(x)}\leq\frac{C}{\abs{x}^2}.
\end{equation*}
Then, by adapting the usual Calder\'on--Zygmund decomposition of $f\in L^1(\R;Y)$, applying the $L^p(\R;Y)$-to-$L^0(\R)$ boundedness to the good part, and the kernel bounds to the bad part, we can deduce that $T^*$ is bounded from $L^1(\R;Y)$ into $L^0(\R)$; cf. a similar argument in \cite[proof of Theorem~C]{TZ}.

The continuous sublinear operator $T^*:L^1(\R;Y)\to L^0(\R)$ is both translation and dilation invariant. Hence, in fact, it is bounded from $L^1(\R;Y)$ to $L^{1,\infty}(\R)$. For $Y=\R$, this is a classical result found for example in \cite[Corollary VI.2.9]{GCRdF}. That it extends to the Banach space -valued context has been observed by Mart\'inez et al. \cite[Lemma 7.3]{MTX}. They formula the result in $L^p(\R;Y)$ under the assumption that $Y$ has Rademacher-type $p$. Here we use the $p=1$ case, which is valid for every Banach space.

Observe that if the Fourier transform of $f\in L^1(\R;Y)$ is compactly supported, then for large enough $n$, the function $S_{0,n}f(x)$ is just $\int_0^\infty\hat{f}(\xi)e^{i2\pi x\cdot\xi}\ud\xi$, the projection of $f$ onto positive frequencies. By the density in $L^1(\R;Y)$ of such functions, this projection, and hence the Hilbert transform, is bounded from $L^1(\R;Y)$ to $L^{1,\infty}(\R;Y)$. This implies the boundedness on $L^p(\R;Y)$ for $p\in(1,\infty)$, and hence the UMD property.
\end{proof}

Let us next check that Theorem~\ref{thm:main} implies its periodic analogue for the Fourier coefficients and their partial sums
\begin{equation*}
  \hat f (k) := \int_{\T} f (x) e ^{-2 \pi i x \cdot k} \ud x,\qquad
  s_{m,n} f(x):=\sum_{k=m}^n \hat{f}(k)e ^{2 \pi i x \cdot k}.
\end{equation*}

\begin{theorem}\label{thm:T}
Let $Y$ be an intermediate UMD space, $p\in(1,\infty)$, and $f\in L^p( \T ;Y)$. Then
\begin{equation*}
  s_{m,n} f(x)\to f(x)
\end{equation*}
as $m\to-\infty,n\to\infty$ for a.e.~$x\in \T $.
In fact,  the  maximal partial sum operator $s^*$,
\begin{equation*}
  s^*f(x):=\sup_{\substack{m,n\in\Z\\ m<n}}\abs{s_{m,n} f(x)},
\end{equation*}
is bounded from $L^p( \mathbb T ;Y)$ to $L^p( \mathbb T)$.
\end{theorem}

\begin{proof}
This result can be reduced to Theorem~\ref{thm:main} by a standard transference of Fourier multipliers. Indeed, consider the $\ell^\infty(\Z^2)$-valued function
\begin{equation*}
  M(\xi):=(1_{[m-1/3,n+1/3]}(\xi))_{m,n\in\Z},
\end{equation*}
which we identify with an operator in $\bddlin(Y,\ell^2(\Z^2;Y))$. Let
\begin{equation*}
  T_M f(x):=\int_{\R}M(\xi)\hat{f}(\xi)e^{i2\pi x\cdot\xi}\ud\xi
  =(S_{m-1/3,n+1/3}f(x))_{m,n\in\Z}
\end{equation*}
be the associated Fourier multiplier operator. Clearly $\abs{T_M f(x)}_{\ell^\infty(\Z^2;Y)}\leq S^* f(x)$, and hence by Theorem~\ref{thm:main}, $T_M$ is bounded from $L^p(\R;Y)$ to $L^p(\R:\ell^\infty(\Z^2;Y))$.

Since $M$ is continuous at the integer points, it follows by standard transference that its restriction to $\Z$ defines a multiplier of Fourier series $\tilde{T}_M$ from $L^p(\T;Y)$ to $L^p(\T;\ell^\infty(\Z^2;Y))$. But
\begin{equation*}
  \abs{\tilde{T}_Mf(x)}_{\ell^\infty(\Z^2;X)}=\abs{(s_{m,n}f(x))_{m,n\in\Z}}_{\ell^\infty(\Z^2;X)}=s^*f(x),
\end{equation*}
and hence $s^*:L^p(\T;Y)\to L^p(\T)$, as claimed.
\end{proof}

\section{Basic notions of time-frequency analysis}

We work in the phase plane $\R\times\R$, where the horizontal axis is thought of as \emph{time} and the vertical axis as \emph{frequency}. A dyadic grid
\begin{equation*}
	\mathscr D = \mathscr D_{t,r}= \{ t+ [n 2 ^{k}r, (n+1) 2 ^{k}r) \;:\; n,k \in \mathbb Z \}
\end{equation*}
is any translation and dilation of the standard dyadic grid, where $ t \in \mathbb R $, $r>0$ are unspecified, and could change from time to time. We use one dyadic grid $\mathscr D=\mathscr D_{t,r}$ for the time axis, and another, $\mathscr{D}'=\mathscr{D}_{t',1/r}$, for the frequency. Observe that the translation parameters $t,t'$ are independent, while the dilation parameters are reciprocals of each other.

A \emph{tile} is a dyadic rectangle $P\subset\R\times\R$ of area $1$, i.e.,
\begin{equation*}
	P=I\times\omega \in\mathscr{D}\times\mathscr{D}',\quad \lvert  I\rvert \cdot \lvert  \omega \rvert=1  .
\end{equation*}
Let $ c (\omega )$ be the center of $ \omega $, and define $ \omega  _{u}= \omega \cap [ c (\omega ) , \infty )$ to 
be the upper half of $ \omega $, and $ P _{u} = I \times \omega $ is the upper half of $ P$. Define the lower half 
of $ \omega$ and $ P$ to be $ \omega _{d }$ and $ P _{d }$ in a similar fashion.  
A tile is written as $ P=I_P \times \omega _P$.    The rectangles  $ P _{d } ,  P _{u}$ will be referred to as half-tiles. 

We define the modulation, translation,  and dilation operators by 
\begin{align*}
	\textup{Mod} _{\lambda } \phi (x) &:= \operatorname e ^{2 \pi i x \cdot \lambda } \phi (x) \,, \qquad \lambda \in \mathbb R \,,
	\\
	T _{t} \phi (x) &:= \phi (x-t)\,, \qquad t\in \mathbb R \,, 
	\\
	\textup{Dil} _{\delta } \phi (x) &:= \delta ^{-1/p} \phi (x/ \delta )\,, \qquad \delta >0\,,\ 0< p < \infty \,. 
\end{align*}
Fix a Schwarz function $ \varphi $ with $ \widehat \varphi \ge0 $, supported in $ [-1/20,1/20]$ so that 
\begin{equation*}
	\sum_{ n \in \mathbb Z } \widehat \varphi (\xi + n/20) ^2  \equiv C \,.  
\end{equation*}
It is known that in this case, the functions $ \{ T _{20n} \varphi \;:\; n \in \mathbb Z \}$ are pairwise orthogonal. 
Indeed, 
\begin{align*}
	\langle T _{20n} \varphi , \varphi \rangle & = \langle  \textup{Mod} _{-20n} \widehat \varphi , \widehat \varphi  \rangle
	\\
	&= \sum_{j \in \mathbb Z } \int _{0} ^{1/20} \operatorname e ^{-40 \pi in x}  \lvert \widehat \varphi  (\xi + j/20)\rvert\ud \xi 
	\\&
	= \int _{0} ^{1/20} \operatorname e ^{-40 \pi in x}   \ud \xi=0 \,. 
\end{align*}
This is a convenience to us at a couple of points below.

To each tile $ P= I \times \omega $, associate a function, a wave packet, given by 
\begin{equation*}
	\varphi_{P} := \textup{Mod} _{c (\omega _{d })} T _{c (I)} \textup{Dil} _{\lvert  I\rvert } ^2 \varphi 
\end{equation*}
Observe that ${\widehat \varphi} _{P}   $ is supported on the interval $ \omega _{P d }$. 

Consider the operator 
\begin{equation} \label{e.CN}
	C_{N (x)} f (x) := \sum_{ P \textup{ is a tile}}  \langle f , \varphi_{P} \rangle \varphi_{P} \mathbf 1_{\omega  _{ P, u}} (N (x)) \,,
\end{equation}
where $ N \;:\; \mathbb R \mapsto \mathbb R $ is a measurable function.   Frequently, the role 
of the measurable function $ N $ is suppressed, and we write $ C_N = C$, and refer to this as the Carleson operator. 

As is explained in \cite[Section 2]{LT:MRL}, to prove Theorem~\ref{thm:main}, it suffices to obtain the estimate
\begin{equation*}
	\Norm{Cf}{L^p(\R_+;X)}\lesssim\Norm{f}{L^p(\R_+;X)}\,, \qquad p\in(1,\infty)\,. 
\end{equation*}
First, we make the standard reduction: by interpolation, it suffices to prove the bound
\begin{equation*}
  \Norm{Cf}{L^{p,\infty}(\R_+;X)}\lesssim\Norm{f}{L^{p,1}(\R_+;X)}
\end{equation*}
for all $p\in(1,\infty)$, which by duality and a well-known description of the Lorentz space $L^{p,1}$ is equivalent to
\begin{equation*}
  \abs{\pair{Cf}{g}}\lesssim\abs{F}^{1/p}\abs{E}^{1/p'}
\end{equation*}
for all $f\in L^\infty(F;X)$, $g\in  L^\infty(E;X^*)$ bounded by one, and all bounded measurable sets $E$ and $F$. Yet another standard reduction, which we recall later, will be used for smaller values of $p>1$.


Central to the proof is the concept of Fourier tile-type, which we discuss after recalling the notion of trees in the following section.

%

\section{Trees}

Henceforth, we will be working with different collections of tiles $ \mathbb P$. It is always assumed that 
for $ P \neq P'\in \mathbb P$, if $ \omega _P = \omega _{P'}$, then $ I_P= I _{P'}+ 20 n \lvert  I_P\rvert $
for some integer $ n$.  This implies that $ \langle \phi _P , \phi _{P'}  \rangle=0$.

A partial order (among dyadic rectangles of equal area) is defined by
\begin{equation*}
\begin{split}
  P\leq P'\ &\overset{\operatorname{def}}{\Leftrightarrow}\ I_P\subseteq I_{P'}\ \text{and}\ \omega_{P}\supseteq\omega_{P'} \\
  &\Leftrightarrow\ P_d\leq P_d'\ \text{or}\ 
     P_u\leq P_u'.
\end{split}
\end{equation*}
We also define
\begin{equation*}
   P\leq_j P'\ \overset{\operatorname{def}}{\Leftrightarrow}\ P_j\leq P_j',\qquad j\in\{d,u\}.
\end{equation*}

A tree $\mathbb{T}$ is a collection of tiles $P$ for which there exists a top tile $T$ (not necessarily an element of $\mathbb{T}$) such that
\begin{equation*}
  P\leq T\qquad\forall P\in\mathbb{T}.
\end{equation*}
Down-trees and up-trees are defined similarly by replacing $\leq$ by $\leq_d$ or $\leq_u$.

The top of a tree $\mathbb{T}$ is not uniquely determined; however, there exists a unique minimal top time interval, which we sometimes denote by $I_{\mathbb{T}}$. Thus $\mathbb{T}$ has at least one top $T$ with $I_T=I_{\mathbb{T}}$, and every other top $T'$ satisfies $I_{T'}\supseteq I_{\mathbb{T}}$.

For a tree $\mathbb{T}$, we define an associated operator by
\begin{equation*}
	A _{\mathbb T} f := \sum_{P\in \mathbb T} \langle f, \phi _P \rangle \phi _P 
\end{equation*}

Basic facts concerning up-trees are collected here. 
\begin{proposition}\label{prop:czo} 
\begin{enumerate}
\item 	Let $ \mathbb T$ be an up-tree. Then the operator
\begin{equation*}
  \textup{Mod} _{-c (\omega _{\mathbb T})} A _{\mathbb T} \textup{Mod} _{c (\omega _{\mathbb T})}
\end{equation*}
is  an $ L ^2 $-bounded Calder\'on-Zygmund operator. 

\item  Let $ 1<p< \infty $.  The UMD property of $X$ is characterized by the property 
	that for any up-tree $ \mathbb T$ the operator $ A _{\mathbb T}$ is bounded from $ L ^{p} (\mathbb R ; X)$ to itself. 
	
\item Let $ X$ be UMD.  For any up-tree and choice of unimodular constants $ \varepsilon _P$, there holds 
	\begin{equation*}
		\Bigl\lVert  \sum_{P\in \mathbb T}  \varepsilon _P \langle f, \phi _P \rangle \phi _P  \Bigr\rVert_{L ^{p} (\mathbb R ; X)} 
		\lesssim 
				\Bigl\lVert  \sum_{P\in \mathbb T} \langle f, \phi _P \rangle \phi _P  \Bigr\rVert_{L ^{p} (\mathbb R ; X)} 
\end{equation*}
\end{enumerate}
\end{proposition}

\begin{proof}
	(1) 
	It suffices to consider the case of a choice of dyadic grid and up-tree $ \mathbb T$ with $ 0= c (\omega _{\mathbb T})$. 
	Write the operator below in kernel form. 
	\begin{align*}
		A _{\mathbb T}f &= \sum_{P\in \mathbb T}  \langle  f  , \varphi_{P} \rangle 
\varphi_{P}  
\\
&= \int  \sum_{P\in \mathbb T}   \varphi_{P}  (x) \overline { \varphi_{P} (y)} f (y)\; dy 
=: \int K (x,y) f (y) \; dy 
\end{align*}
It is routine to check that $ K (x,y) $ is a Calder\'on-Zygmund kernel  with the size and smoothness conditions, 
namely 
\begin{equation*}
	\bigl\lvert D ^{j} K (x,u)\bigr\rvert \lesssim \lvert  x-y\rvert ^{-j-1}\,, \qquad x\neq y, \ j\in \mathbb N \,.  
\end{equation*}
(This also holds for down-trees.)  To prove the $ L ^2 $ bound for up-trees, note that the functions $ \{\phi _P \;:\; P\in \mathbb T\}$ 
are pairwise orthogonal, and have constant $ L ^2 $-norm. 

\medskip

(2)  If $ X$ is UMD, then $ A _{\mathbb T}$ is an  Calder\'on-Zygmund operator conjugated by an exponential. 
Hence, the $ A _{\mathbb T}$ is bounded on $ L ^{p} (\mathbb R ; X)$, with constant depending only upon $ X$ and $ p$. 
In the reverse direction,  the argument in \cite[Section 2]{LT:MRL} shows that 
the operators $ \Pi_{\xi } $ with  $ ( \Pi  _{\xi } f) = \mathbf 1_{(- \infty , \xi )} \widehat f $, are the appropriate averages of 
the up-tree operators $ A _{\mathbb T}$. 
Hence, the Hilbert transform is bounded on $ L ^{p} (\mathbb R ; X)$, showing that $ X$ must be UMD. 

\medskip 

(3) Let $ B _{\mathbb T} f =  \sum_{P\in \mathbb T}  \varepsilon _P \langle f, \phi _P \rangle \phi _P $, and assume that 
$ 0 \in \omega _{T,u}$.  This is a Calder\'on-Zygmund operator, hence it acts boundedly on $ L ^{p} (\mathbb R ; X)$. 
Since the functions $ \{\varphi _P \;:\; P\in \mathbb T\}$ 
are pairwise orthogonal, 
\begin{equation*}
\lVert  B _{\mathbb T} f \rVert_{L ^{p} (\mathbb R ; X)} 
= \lVert  B _{\mathbb T} A _{\mathbb T} f \rVert_{L ^{p} (\mathbb R ; X)} 
\lesssim \lVert  A _{\mathbb T} f \rVert_{L ^{p} (\mathbb R ; X)} \,. \qedhere
\end{equation*}
\end{proof}


\section{The Fourier tile-type of a Banach space: generalities}


In this section we introduce the notion of quasi-orthogonality in the phase plane that will be the key to our control of the Carleson operator. We also make some simple observations about this notion. Our main examples of spaces in which this quasi-orthogonality holds will be studied in the subsequent sections.

\begin{definition}
We say that a Banach space $X$ has \emph{Fourier tile-type} $(q,\alpha)$ if there is a constant $C$ such that
\begin{equation}\label{eq:FtileType}
\begin{split}
  \Big(\sum_{\mathbb{T}\in\mathscr{T}} &\BNorm{\sum_{P\in\mathbb{T}}\pair{f}{\phi_P}\phi_P}{L^q(\R;X)}^q\Big)^{1/q} \\
  &\leq C\Norm{f}{L^q(\R;X)}+C\Big(\Norm{f}{L^\infty(\R;X)}\Big[\sum_{\mathbb{T}\in\mathscr{T}}\abs{I_{\mathbb{T}}}\Big]^{1/q}\Big)^{1-\alpha}\Norm{f}{L^q(\R;X)}^{\alpha}
\end{split}
\end{equation}
whenever $f\in L^q(\R;X)\cap L^\infty(\R;X)$ and $\mathscr{T}$ is a finite collection of finite trees with the following property:
\begin{equation}\label{eq:disjProp}
  \textup{If $P\in\mathbb{T}\in\mathscr{T}$ and $P'\in\mathbb{T}'\in\mathscr{T}$ satisfy $\omega_P\subseteq\omega_{P_d'}$, then $I_{P'}\cap I_{\mathbb{T}}=\varnothing$.}
\end{equation}
We say that a Banach space has Fourier tile-type $q$ if it has Fourier tile-type $(q,\alpha)$ for every $\alpha\in(0,1)$.
\end{definition}

\begin{remark}
For a fixed $q$, the condition of Fourier tile-type $(q,\alpha)$ becomes more restrictive with increasing $\alpha$, so the point of Fourier tile-type $q$ is that we can take $\alpha$ as close to $1$ as we like.


In the Walsh model, we defined \cite{HL:Walsh} \emph{tile-type} (which could be more precisely called the \emph{Walsh tile-type}) by the requirement that \eqref{eq:FtileType} hold with the Walsh wave packets $w_P$ in place of $\phi_P$, and with $\alpha=1$. With the methods of the paper at hand, one can easily see that all our results proved under the assumption of Walsh tile-type remain valid with a similarly relaxed notion as here.
\end{remark}

\begin{lemma}
Property~\eqref{eq:disjProp} implies that all down-halves $P_d$ of all
\begin{equation*}
   P\in\bigcup_{\mathbb{T}\in\mathscr{T}}\mathbb{T} =:\mathbb{P}
\end{equation*}
are pairwise disjoint.
\end{lemma}

\begin{proof}
Let $P\in\mathbb{T}$, $P'\in\mathbb{T}'$ be two different tiles in $\mathbb{P}$. If $\omega_{P_d}\cap\omega_{P_d'}=\varnothing$, then $P_d\cap P_d'=\varnothing$. So assume that $\omega_{P_d}\cap\omega_{P_d'}\neq\varnothing$, and then for example $\omega_{P_d}\subseteq\omega_{P_d'}$. If $\omega_{P_d}=\omega_{P_d'}$ for two tiles, then $I_P$ and $I_{P'}$ (which have equal size) are either equal (which cannot be in our case, since this would imply that $P=P'$) or disjoint (in which case $P\cap P'=\varnothing$). We are left with the case that $\omega_{P_d}\subsetneq\omega_{P_d'}$, and then $\omega_{P}\subsetneq \omega_{P_d'}$. But this implies that $I_{P'}\cap I_{\mathbb{T}}=\varnothing$ by \eqref{eq:disjProp}, hence $I_{P'}\cap I_P=\varnothing$ and thus $P\cap P'=\varnothing$ also in this final case.
\end{proof}

\begin{lemma}\label{lem:disjProp}
Suppose that $\mathscr{T}$ satisfies \eqref{eq:disjProp}, and fix a $P\in\mathbb{T}\in\mathscr{T}$. Then among the tiles $P'\in\mathbb{P}$ with $\omega_{P_d'}\supseteq\omega_P$, the time intervals $I_{P'}$ are pairwise disjoint and contained in $I_{\mathbb{T}}^c$.
\end{lemma}

\begin{proof}
Let $P',P''$ be two such tiles. So in particular $\omega_{P_d'}\cap\omega_{P_d''}\supseteq\omega_P\neq\varnothing$. Since $P_d'\cap P_d''=\varnothing$, it must be that $I_{P'}\cap I_{P''}=\varnothing$. The fact that $I_{P'}\subseteq I_{\mathbb{T}}^c$ is immediate from \eqref{eq:disjProp}
\end{proof}

\begin{lemma}\label{lem:upTreesOnly}
Property~\eqref{eq:disjProp} implies that each $\mathbb{T}\in\mathscr{T}$ can be divided into up-trees $\mathbb{T}_j$, $j\in J(\mathbb{T})$, with pairwise disjoint supports $I_{\mathbb{T}_j}$. Hence, writing
\begin{equation*}
  \tilde{\mathscr{T}}:=\{\mathbb{T}_j:j\in J(\mathbb{T}),\mathbb{T}\in\mathscr{T}\},
\end{equation*}
inequality \eqref{eq:FtileType} for $\tilde{\mathscr{T}}$ implies \eqref{eq:FtileType} for $\mathscr{T}$.
In particular, it suffices to consider the Fourier-tile type condition for up-trees only.
\end{lemma}

\begin{proof}
Let $T_j$, $j\in J(\mathbb{T})$ be the maximal tiles in $\mathbb{T}$. Then $\mathbb{T}$ splits into the trees $\mathbb{T}_j:=\{P\in\mathbb{T}:P\leq T_j\}$. From \eqref{eq:disjProp} with $P'=T_j$ it follows that each $\mathbb{T}_j$ is an up-tree. Also, since $\omega_{\mathbb{T}_j}\cap\omega_{\mathbb{T}_{j'}}\supseteq\omega_{\mathbb{T}}\neq\varnothing$, it follows that $I_{T_j}\cap I_{T_{j'}}=\varnothing$ since $T_j\cap T_{j'}=\varnothing$ by maximality.
\end{proof}

\section{The Fourier tile-type of a Hilbert space}

The starting point of the deeper aspect of Fourier tile-type is the validity of the defining estimate in a Hilbert space. We start with a preliminary formulation of this result.

\begin{proposition}\label{prop:HilbertBasicEst}
For a Hilbert space $H$, the following estimate holds:
\begin{equation*}
\begin{split}
  \Big(\sum_{P\in\mathbb{P}} &\abs{\pair{f}{\phi_P}}^2\Big)^{1/2} \\
  &\leq C\Norm{f}{L^2(\R;X)}+C\Big(\sup_{P\in\mathbb{P}}\frac{\abs{\pair{f}{\phi_P}}}{\abs{I_P}^{1/2}}
    \Big[\sum_{\mathbb{T}\in\mathscr{T}}\abs{I_{\mathbb{T}}}\Big]^{1/2}\Big)^{1/3}\Norm{f}{L^2(\R;H)}^{2/3}
\end{split}
\end{equation*}
whenever $f\in L^2(\R;H)$, and $\mathbb{P}=\bigcup_{\mathbb{T}\in\mathscr{T}}\mathbb{T}$, where $\mathscr{T}$ is as in \eqref{eq:disjProp}.
\end{proposition}

\begin{remark}
Proposition~\ref{prop:HilbertBasicEst} is essentially contained in the proof of Proposition~3.2 in \cite{LT:MRL}, but not explicitly formulated as here, so we reproduce the proof for completeness.
\end{remark}

\begin{proof}[Proof of Proposition~\ref{prop:HilbertBasicEst}]
We denote the left side by $S$, and estimate
\begin{equation*}
\begin{split}
  S^2:=\sum_{P\in\mathbb{P}}\abs{\pair{f}{\phi_P}}^2
  &=\Bpair{\sum_{P\in\mathbb{P}}\pair{f}{\phi_P}\phi_P}{f} \\
  &\leq\BNorm{\sum_{P\in\mathbb{P}}\pair{f}{\phi_P}\phi_P}{L^2(\R;H)}\Norm{f}{L^2(\R;H)}.
\end{split}
\end{equation*}
Here
\begin{equation*}
\begin{split}
  \BNorm{\sum_{P\in\mathbb{P}}\pair{f}{\phi_P}\phi_P}{L^2(\R;H)}^2
  &=\sum_{P,P'\in\mathbb{P}}\pair{f}{\phi_P}\pair{\phi_P}{\phi_{P'}}\pair{\phi_{P'}}{f} \\
  &=\Big(\sum_{\substack{P,P'\in\mathbb{P}\\ \omega_P=\omega_{P'} }}
    +2\sum_{\substack{P,P'\in\mathbb{P}\\ \omega_P\subseteq\omega_{P'_d} }}\Big)\pair{f}{\phi_P}\pair{\phi_P}{\phi_{P'}}\pair{\phi_{P'}}{f} \\
   &=:S_1+2S_2,
\end{split}
\end{equation*}
where the middle line follows from the fact that $\supp\hat\phi_P\subseteq\omega_{P_d}$, so that $\pair{\phi_P}{\phi_{P'}}\neq 0$ only if $\omega_{P_d}\cap\omega_{P_d}\neq\varnothing$, which means that these intervals either coincide, or one is strictly contained in the other.

To proceed further, we need the elementary estimate
\begin{equation}\label{eq:ipEst}
  \abs{\pair{\phi_P}{\phi_{P'}}}\lesssim\Big(\frac{\abs{I_P}}{\abs{I_{P'}}}\Big)^{1/2}\Norm{v_{I_P}1_{I_{P'}}}{1},
  \qquad \qquad\abs{I_{P'}}\leq\abs{I_P}.
\end{equation}
where $\displaystyle v_I(x):=\frac{1}{\abs{I}}\Big(1+\frac{\abs{x-c(I)}}{\abs{I}}\Big)^{-10}$.

In $S_1$, we can use \eqref{eq:ipEst} with either order of $I_P$ and $I_{P'}$, to get
\begin{equation*}
\begin{split}
  S_1 &\leq \sum_{\substack{P,P'\in\mathbb{P}\\ \omega_P=\omega_{P'} }}\frac{1}{2}(\abs{\pair{f}{\phi_P}}^2+\abs{\pair{f}{\phi_{P'}}}^2)
         \min\{\Norm{v_{I_P}1_{I_{P'}}}{1},\Norm{v_{I_{P'}}1_{I_{P}}}{1}\} \\
       &\leq\sum_{P\in\mathbb{P}}\abs{\pair{f}{\phi_P}}^2\sum_{\substack{P'=I'\times\omega'\in\mathbb{P}\\ \omega'=\omega_{P} }}\Norm{v_{I_P}1_{I'}}{1} 
         \leq\sum_{P\in\mathbb{P}}\abs{\pair{f}{\phi_P}}^2\sum_{\substack{I'\in\mathscr{D}\\ \abs{I'}=\abs{I_P}}}\Norm{v_{I_P}1_{I'}}{1} \\
        &=\sum_{P\in\mathbb{P}}\abs{\pair{f}{\phi_P}}^2\Norm{v_{I_P}}{1}\lesssim\sum_{P\in\mathbb{P}}\abs{\pair{f}{\phi_P}}^2 = S^2.
\end{split}
\end{equation*}
We used the fact that the intervals $I'\in\mathscr{D}$ with $\abs{I'}=\abs{I_P}$ form a partition of~$\R$.

We turn our attention to $S_2$:
\begin{equation*}
\begin{split}
  S_2 &\lesssim\sum_{P\in\mathbb{P}}\abs{\pair{f}{\phi_P}}\sum_{\substack{P'\in\mathbb{P}\\ \omega_{P_d'}\supset\omega_P}}
    \Big(\frac{\abs{I_P}}{\abs{I_{P'}}}\Big)^{1/2}\Norm{v_{I_P}1_{I_{P'}}}{1}\abs{\pair{\phi_{P'}}{f}} \\
    &\leq\Big(\sup_{P'\in\mathbb{P}}\frac{\abs{\pair{\phi_{P'}}{f}}}{\abs{I_{P'}}^{1/2}}\Big)
      \sum_{P\in\mathbb{P}}\abs{\pair{f}{\phi_P}}\abs{I_P}^{1/2} \sum_{\substack{P'\in\mathbb{P}\\ \omega_{P_d'}\supset\omega_P}}\Norm{v_{I_P}1_{I_{P'}}}{1} \\
    &\leq\Big(\sup_{P'\in\mathbb{P}}\frac{\abs{\pair{\phi_{P'}}{f}}}{\abs{I_{P'}}^{1/2}}\Big)
      \sum_{P\in\mathbb{P}}\abs{\pair{f}{\phi_P}}\abs{I_P}^{1/2}\Norm{v_{I_P}1_{I_{\mathbb{T}(P)}^c}}{1}.
\end{split}
\end{equation*}
In the last line, we denoted by $\mathbb{T}(P)$ the the unique tree with $P\in\mathbb{T}(P)\in\mathscr{T}$, and used Lemma~\ref{lem:disjProp} which guarantees that the intervals $I_{P'}$ appearing in the inner sum on the penultimate line are pairwise disjoint and contained in $I_{\mathbb{T}(P)}^c$.

We use Cauchy--Schwarz to get
\begin{equation*}
\begin{split}
   \sum_{P\in\mathbb{P}} &\abs{\pair{f}{\phi_P}}\abs{I_P}^{1/2}\Norm{v_{I_P}1_{I_{\mathbb{T}(P)}^c}}{1} \\
   &\leq\Big(\sum_{P\in\mathbb{P}}\abs{\pair{f}{\phi_P}}^2\Big)^{1/2}
      \Big(\sum_{P\in\mathbb{P}}\abs{I_P}\Norm{v_{I_P}1_{I_{\mathbb{T}(P)}^c}}{1}^2\Big)^{1/2} \\
   &\lesssim S\Big(\sum_{P\in\mathbb{P}}\abs{I_P}\Norm{v_{I_P}1_{I_{\mathbb{T}(P)}^c}}{1}\Big)^{1/2},
\end{split}
\end{equation*}
where we estimated $\Norm{v_{I_P}1_{I_{\mathbb{T}(P)}^c}}{1}\leq\Norm{v_{I_P}}{1}\lesssim 1$.

Finally, we write
\begin{equation*}
\begin{split}
   \sum_{P\in\mathbb{P}}\abs{I_P}\Norm{v_{I_P}1_{I_{\mathbb{T}(P)}^c}}{1}
   &=\sum_{\mathbb{T}\in\mathscr{T}}\sum_{P\in\mathbb{T}}\abs{I_P}\Norm{v_{I_P}1_{I_{\mathbb{T}}^c}}{1} \\
   &\leq\sum_{\mathbb{T}\in\mathscr{T}}\sum_{P=I\times\omega\leq I_{\mathbb{T}}\times\omega_{\mathbb{T}}}
        \abs{I}\Norm{v_I 1_{I_{\mathbb{T}}^c}}{1} 
   \leq\sum_{\mathbb{T}\in\mathscr{T}}\sum_{\substack{I\in\mathscr{D}\\ I\subseteq I_{\mathbb{T}}}}  \abs{I}\Norm{v_I 1_{I_{\mathbb{T}}^c}}{1},
\end{split}
\end{equation*}
where the last step follows from the fact that the summand depends only on the time interval $I$, and that the frequency interval $\omega$ of $P=I\times\omega$ is uniquely determined by $I$, since $\abs{\omega}=1/\abs{I}$ and $\omega\supseteq\omega_{\mathbb{T}}$. Using the elementary bound
\begin{equation*}
     \sum_{\substack{I\in\mathscr{D}\\ I\subseteq I_{\mathbb{T}}}}   \abs{I}\Norm{v_I 1_{I_{\mathbb{T}}^c}}{1}\lesssim\abs{I_{\mathbb{T}}},
\end{equation*}
and combining all the estimates, we have shown that
\begin{equation*}
  S^2\leq\sqrt{S_1+2S_2}\Norm{f}{2}
  \lesssim\sqrt{S^2+AS}\Norm{f}{2},
\end{equation*}
where
\begin{equation*}  
   A:=\Big(\sup_{P'\in\mathbb{P}}\frac{\abs{\pair{\phi_{P'}}{f}}}{\abs{I_{P'}}^{1/2}}\Big)
    \Big(\sum_{\mathbb{T}\in\mathscr{T}}\abs{I_{\mathbb{T}}}\Big)^{1/2}\Norm{f}{2}.
\end{equation*}
If $S^2\geq AS$, we get $S^2\lesssim S\Norm{f}{2}$, and hence $S\lesssim\Norm{f}{2}$. If $S^2<AS$, then $S^2\lesssim A^{1/2}S^{1/2}\Norm{f}{2}$, thus $S^{3/2}\lesssim A^{1/2}\Norm{f}{2}$, and hence $S\lesssim A^{1/3}\Norm{f}{2}^{2/3}$. So in any case we deduce that
\begin{equation*}
  S\lesssim \Norm{f}{2}+A^{1/3}\Norm{f}{2}^{2/3},
\end{equation*}
which is the asserted bound.
\end{proof}

\begin{corollary}\label{cor:HilbertWeakType}
For a Hilbert space $H$, the following estimate holds:
\begin{equation*}
  \sum_{\substack{P\in\mathbb{P}: \\ \frac{\abs{\pair{f}{\phi_P}}}{\sqrt{\abs{I_P}}}>\lambda}} \abs{I_P}
  \leq \frac{C}{\lambda^2}\Norm{f}{L^2(\R;H)}^2.
\end{equation*}
whenever $f\in L^2(\R;H)$, and $\mathbb{P}=\bigcup_{\mathbb{T}\in\mathscr{T}}\mathbb{T}$, where $\mathscr{T}$ is as in \eqref{eq:disjProp}.
\end{corollary}

\begin{remark}
This result is from \cite{BL:03}. It says that the mapping
\begin{equation*}
  f\mapsto\Big\{\frac{\abs{\pair{f}{\phi_P}}}{\sqrt{\abs{I_P}}}1_{I_P}:P\in\mathbb{P}\Big\}
\end{equation*}
takes $L^2(\R;H)$ boundedly into $L^{2,\infty}(\R\times\mathbb{P};H)$, where $\R\times\mathbb{P}$ is equipped with the product of the Lebesgue and the counting measures. We prove it as a corollary to Proposition~\ref{prop:HilbertBasicEst} for completeness.
\end{remark}

\begin{proof}
Let $S_\lambda$ stand for the left side as written, and $S_\lambda'$ for a modification with the summation condition $\lambda<\frac{\abs{\pair{f}{\phi_P}}}{\sqrt{\abs{I_P}}}\leq 2\lambda$ instead. Thus $S_\lambda=\sum_{k=0}^\infty S_{2^k\lambda}'$. If we can prove the assertion for $S_\lambda'$ in place of $S_\lambda$, then
\begin{equation*}
  S_\lambda\leq \sum_{k=0}^\infty C(2^k\lambda)^{-2}\Norm{f}{L^2(\R:H)}^2 =\frac{4}{3}\frac{C}{\lambda^2}\Norm{f}{L^2(\R^2;H)}^2,
\end{equation*}
so we obtain the claim as stated.

To prove the bound for $S_\lambda'$, denote
\begin{equation*}
  \mathbb{P}_\lambda:=\{P\in\mathbb{P}: \lambda<\frac{\abs{\pair{f}{\phi_P}}}{\sqrt{\abs{I_P}}}\leq 2\lambda\},
\end{equation*}
and write this as a union of single-tile trees
\begin{equation*}
  \mathbb{P}_\lambda=\bigcup_{\mathbb{T}\in\mathscr{T}_\lambda}\mathbb{T},\qquad \mathscr{T}_\lambda:=\{\{P\}:P\in\mathbb{P}_\lambda\}.
\end{equation*}
Since $\mathbb{P}_\lambda\subseteq\mathbb{P}$, it is immediate that $\mathscr{T}_\lambda$ also satisfies \eqref{eq:disjProp}. Hence, from Proposition~\ref{prop:HilbertBasicEst}, we deduce that
\begin{equation*}
\begin{split}
  \lambda^2\sum_{P\in\mathbb{P}_\lambda}\abs{I_P}
  &\leq \sum_{P\in\mathbb{P}_\lambda}\abs{\pair{f}{\phi_P}}^2 \\
  &\leq C\Norm{f}{L^2(\R;H)}^2+C\Big(\sup_{P\in\mathbb{P}_\lambda}\frac{\abs{\pair{f}{\phi_P}}}{\abs{I_P}^{1/2}}
    \Big[\sum_{P\in\mathbb{P}_\lambda}\abs{I_P}\Big]^{1/2}\Big)^{2/3}\Norm{f}{L^2(\R;H)}^{4/3} \\
  &\leq C\Norm{f}{L^2(\R;H)}^2+C\Big(2\lambda
    \Big[\sum_{P\in\mathbb{P}_\lambda}\abs{I_P}\Big]^{1/2}\Big)^{2/3}\Norm{f}{L^2(\R;H)}^{4/3}.
\end{split}
\end{equation*}
Thus
\begin{equation*}
  \lambda^2 S_\lambda'=\lambda^2\sum_{P\in\mathbb{P}_\lambda}\abs{I_P}\leq C\Norm{f}{L^2(\R:H)}^2
    +C\big(\lambda^2 S_\lambda'\big)^{1/3}\Norm{f}{L^2(\R;H)}^{4/3},
\end{equation*}
from which the claim that $\lambda^2 S_\lambda'\leq C\Norm{f}{L^2(\R;H)}^2$ immediately follows.
\end{proof}

A combination of the previous results leads to the final form of the tile-type inequality for Hilbert spaces:

\begin{proposition}\label{prop:HilbertTileType}
Every Hilbert space has Fourier tile-type $2$. In fact, the following more precise estimate is valid:
\begin{equation*}
\begin{split}
  \Big(\sum_{\mathbb{T}\in\mathscr{T}} &\BNorm{\sum_{P\in\mathbb{T}}\pair{f}{\phi_P}\phi_P}{L^2(\R;H)}^2\Big)^{1/2}
   \leq C\Big(\sum_{P\in\mathbb{P}} \abs{\pair{f}{\phi_P}}^2\Big)^{1/2} \\
  &\leq C\Norm{f}{L^2(\R;H)}\Big\{1+\log_+\Big(\frac{\Norm{f}{L^\infty(\R;H)}}{\Norm{f}{L^2(\R;H)}}\Big[\sum_{\mathbb{T}\in\mathscr{T}}\abs{I_{\mathbb{T}}}\Big]^{1/2}\Big)\Big\}^{1/2} \\
  &\leq C\Norm{f}{L^2(\R;H)}+\frac{C}{\sqrt{1-\alpha}}
     \Big(\Norm{f}{L^\infty(\R;H)}\Big[\sum_{\mathbb{T}\in\mathscr{T}}\abs{I_{\mathbb{T}}}\Big]^{1/2}\Big)^{1-\alpha}\Norm{f}{L^2(\R;H)}^{\alpha}.
\end{split}
\end{equation*}
whenever $f\in L^2(\R;H)\cap L^\infty(\R;H)$, and $\mathbb{P}=\bigcup_{\mathbb{T}\in\mathscr{T}}\mathbb{T}$, where $\mathscr{T}$ is as in \eqref{eq:disjProp}. Here $C$ is independent of $\alpha\in(0,1)$.
\end{proposition}

\begin{proof}
By Lemma~\ref{lem:upTreesOnly}, we may assume that all $\mathbb{T}\in\mathscr{T}$ are up-trees. Then, for each fixed $\mathbb{T}$, the functions $\phi_P$, $P\in\mathbb{T}$, split into 20 pairwise orthonormal subcollections. Thus
\begin{equation*}
   \BNorm{\sum_{P\in\mathbb{T}}\pair{f}{\phi_P}\phi_P}{L^2(\R;H)}^2
  \lesssim\sum_{P\in\mathbb{T}}\abs{\pair{f}{\phi_P}}^2,
\end{equation*}
which gives the first estimate. The third estimate is elementary.

The rest of the proof is concerned with the second estimate. We first observe the upper bound
\begin{equation*}
  \frac{\abs{\pair{f}{\phi_P}}}{\sqrt{\abs{I_P}}}\leq \frac{\Norm{f}{\infty}\Norm{\phi_P}{1}}{\sqrt{\abs{I_P}}}\leq C\Norm{f}{\infty}.
\end{equation*}
Accordingly, we can split
\begin{equation*}
  \mathbb{P}:=\bigcup_{k=-1}^K\mathbb{P}_k,
\end{equation*}
where
\begin{equation*}
\begin{split}
  \mathbb{P}_{-1} &:=\Big\{P\in\mathbb{P}:\frac{\abs{\pair{f}{\phi_P}}}{\sqrt{\abs{I_P}}}\leq\lambda
     :=\frac{\Norm{f}{2}}{\sqrt{\sum_{\mathbb{T}\in\mathscr{T}}\abs{I_{\mathbb{T}}}}}\Big\},\\
  \mathbb{P}_k &:=\Big\{P\in\mathbb{P}:2^k\lambda<\frac{\abs{\pair{f}{\phi_P}}}{\sqrt{\abs{I_P}}}\leq 2^{k+1}\lambda\Big\},\qquad
      k=0,1,\ldots, K,
\end{split}
\end{equation*}
and
\begin{equation*}
  K:=\log_2^+\Big(C\frac{\Norm{f}{\infty}}{\Norm{f}{2}}\sqrt{\sum_{\mathbb{T}\in\mathscr{T}}\abs{I_{\mathbb{T}}}}\Big).
\end{equation*}
Each $\mathbb{P}_k$ can be written as
\begin{equation*}
  \mathbb{P}_k=\bigcup_{\mathbb{T}\in\mathscr{T}}\mathbb{T}\cap\mathbb{P}_k,
\end{equation*}
where $\mathbb{T}\cap\mathbb{P}_k$ is a tree with the same top as $\mathbb{T}$, and it is immediate that each $\mathscr{T}_k:=\{\mathbb{T}\cap\mathbb{P}_k:\mathbb{T}\in\mathscr{T}\}$ inherits property \eqref{eq:disjProp} form $\mathscr{T}$.

Hence, from Proposition~\ref{prop:HilbertBasicEst} we have that
\begin{equation*}
\begin{split}
  \sum_{P\in\mathbb{P}_{-1}}\abs{\pair{f}{\phi_P}}^2
  &\lesssim\Norm{f}{2}^2+\Big(\sup_{P\in\mathbb{P}_{-1}}\frac{\abs{\pair{f}{\phi_P}}}{\abs{I_P}^{1/2}}\sqrt{\sum_{\mathbb{T}\in\mathscr{T}}\abs{I_{\mathbb{T}}}}\Big)^{2/3}\Norm{f}{2}^{4/3}\\
  &\leq\Norm{f}{2}^2+\Big(\lambda\sqrt{\sum_{\mathbb{T}\in\mathscr{T}}\abs{I_{\mathbb{T}}}}\Big)^{2/3}\Norm{f}{2}^{4/3}
    =\Norm{f}{2}^2+\Norm{f}{2}^{2/3}\Norm{f}{2}^{4/3},
\end{split}
\end{equation*}
and from Corollary~\ref{cor:HilbertWeakType} that
\begin{equation*}
\begin{split}
  \sum_{P\in\mathbb{P}_k}\abs{\pair{f}{\phi_P}}^2
  \lesssim (2^k\lambda)^2\sum_{P\in\mathbb{P}_k}\abs{I_P}
  \lesssim (2^k\lambda)^2\frac{1}{(2^k\lambda)^2}\Norm{f}{2}^2=\Norm{f}{2}^2.
\end{split}
\end{equation*}
Altogether, we obtain
\begin{equation*}
\begin{split}
  \sum_{P\in\mathbb{P}}\abs{\pair{f}{\phi_P}}^2
  &=\sum_{k=-1}^K\sum_{P\in\mathbb{P}_k}\abs{\pair{f}{\phi_P}}^2
  \lesssim(1+K)\Norm{f}{2}^2 \\
  &\lesssim\Norm{f}{2}^2\Big[1+\log_+\Big(\frac{\Norm{f}{\infty}}{\Norm{f}{2}}\sqrt{\sum_{\mathbb{T}\in\mathscr{T}}\abs{I_{\mathbb{T}}}}\Big)\Big].
\end{split}
\end{equation*}
The assertion follows by taking the square root of both sides.
\end{proof}

\section{The Fourier tile-type of interpolation spaces}

In order to extend the class of spaces with Fourier tile-type beyond the Hilbertian realm, we employ interpolation techniques. We start with a technical statement concerning the intersection spaces that appear in the definition of Fourier tile-type.

\begin{lemma}\label{lem:interpolation}
Let $\{X_0,X_1\}$ be an interpolation couple of Banach spaces, and denote $X_\theta:=[X_0,X_1]_\theta$ and $q:=2/\theta$. Let $L^2(\R;X_1)\cap L^\infty(\R;X_1)$ be equipped with the norm
\begin{equation*}
  \Norm{f}{L^2(\R;X_1)\cap L^\infty(\R;X_1)}:=A\Norm{f}{L^\infty(\R;X_1)}+B\Norm{f}{L^2(\R;X_1)},
\end{equation*}
and $L^q(\R;X_\theta)\cap L^\infty(\R;X_\theta)$ with the norm
\begin{equation*}
  \Norm{f}{L^q(\R;X_\theta)\cap L^\infty(\R;X_\theta)}:=A^\theta\Norm{f}{L^\infty(\R;X_\theta)}+B^\theta\Norm{f}{L^q(\R;X_\theta)}.
\end{equation*}
Then
\begin{equation*}
   L^q(\R;X_\theta)\cap L^\infty(\R;X_\theta)
   = [L^\infty(X_0),L^2(\R;X_1)\cap L^\infty(\R;X_1)]_\theta,
\end{equation*}
where the norms are uniformly equivalent, independently of the positive numbers $A$ and $B$.
\end{lemma}

\begin{remark}
It is well known that
\begin{equation*}
\begin{split}
  L^q(\R;X_\theta) &= [L^\infty(\R;X_0),L^2(\R;X_1)]_\theta,\\
  L^\infty(\R;X_\theta) &= [L^\infty(\R;X_0),L^\infty(\R;X_1)]_\theta.
\end{split}
\end{equation*}
However, in general the identity $[E,F\cap G]_\theta=[E,F]_\theta\cap[E,G]_\theta$ seems to be non-trivial. See \cite{Haase:06}, Section 3.3, for comments on the case of real interpolation.
%
\end{remark}

\begin{proof}
Let
\begin{equation*}
  f=\sum_{k=1}^\infty x_k 1_{E_k}
  =\sum_{k=1}^\infty x_k^0\Norm{x_k}{X_\theta} 1_{E_k}\in L^q(\R;X_\theta)\cap L^\infty(\R;X_\theta)
\end{equation*}
be a countably-valued function, where the measurable sets $E_k$ are pairwise disjoint. Such functions are dense in the space under consideration. Suppose further that the norm of $f$ in this space is at most one. Since $x_k^0\in X_\theta\in[X_0,X_1]_\theta$, we can find a function
\begin{equation*}
\begin{split}
  \phi_k\in\mathscr{F}(X_0,X_1):=\Big\{\phi: & [0,1]+i\R\to X_0+X_1;  \\
   &  \phi\textup{ holomorphic on }(0,1)+i\R,\\ 
   & \phi \textup{ continuous and bounded on }[0,1]+i\R,\\
   &\Norm{\phi}{\mathscr{F}(X_0,X_1)}:=\sup_{u\in\{0,1\}}\sup_{t\in\R}\Norm{\phi(u+it)}{X_u}<\infty\Big\}
\end{split}
\end{equation*}
with $\phi_k(\theta)=x_k^0$ and $\Norm{\phi_k}{\mathscr{F}(X_0,X_1)}\lesssim\Norm{x_k^0}{X_\theta}=1$. We define
\begin{equation*}
  F(z):=\sum_{k=1}^\infty \phi_k(z)\Norm{x_k}{X_\theta}^{z/\theta}1_{E_k}.
\end{equation*}
It follows that $  \Norm{F(it)}{L^\infty(\R;X_0)}\lesssim 1$, and
\begin{equation*}
  \Norm{F(1+it)}{L^\infty(\R;X_1)}\lesssim \Norm{f}{L^\infty(\R;X_\theta)}^{1/\theta},\qquad
  \Norm{F(1+it)}{L^2(\R;X_1)}\lesssim \Norm{f}{L^q(\R;X_\theta)}^{1/\theta},
\end{equation*}
so that
\begin{equation*}
\begin{split}
   \Norm{F(1+it)}{L^2(\R;X_1)\cap L^\infty(\R;X_1)}
   &\lesssim A\Norm{f}{L^\infty(\R;X_\theta)}^{1/\theta}+B\Norm{f}{L^q(\R;X_\theta)}^{1/\theta} \\
   &\leq \big(A^\theta\Norm{f}{L^\infty(\R;X_\theta)}+B^\theta\Norm{f}{L^q(\R;X_\theta)}\big)^{1/\theta} \\
   &=\Norm{f}{L^q(\R;X_\theta)\cap L^\infty(\R;X_\theta)}^{1/\theta}\leq 1.
\end{split}
\end{equation*}
Hence
\begin{equation*}
  F\in \mathscr{F}(L^\infty(\R;X_0),L^2(\R;X_1)\cap L^\infty(\R;X_1))
\end{equation*}
with norm $\lesssim 1$, and therefore
\begin{equation*}
  f=F(\theta)\in[L^\infty(\R;X_0),L^2(\R;X_1)\cap L^\infty(\R;X_1))]_\theta
\end{equation*}
with norm $\lesssim 1$. This proves the bounded embedding
\begin{equation*}
  L^q(\R;X_\theta)\cap L^\infty(\R;X_\theta)
  \subseteq [L^\infty(X_0),L^2(\R;X_1)\cap L^\infty(\R;X_1)]_\theta.
\end{equation*}

For the converse direction, it is immediate from standard results that
\begin{equation*}
  [L^\infty(X_0),L^2(\R;X_1)\cap L^\infty(\R;X_1)]_\theta
  \subseteq\begin{cases} [L^\infty(\R;X_0),L^2(\R;X_1)]_\theta= L^q(\R;X_\theta),  \\
      [L^\infty(\R;X_0),L^\infty(\R;X_1)]_\theta= L^\infty(\R;X_\theta), \end{cases}
\end{equation*}
and therefore also that
\begin{equation*}
  [L^2(\R;X_1)\cap L^\infty(\R;X_1),L^\infty(X_0)]_\theta
  \subseteq L^q(\R;X_\theta)\cap L^\infty(\R;X_\theta).
\end{equation*}
For the correct norm estimate, one should note that if $Y_1$ is equipped with $\lambda$ times its usual norm, then $[Y_0,Y_1]_\theta$ will be equipped with $\lambda^\theta$ times its usual norm. This fact is easy to check and completes the proof.
\end{proof}

The main result of this section is the following, which provides a good supply of spaces with non-trivial Fourier tile-type. It is through this result that the class of intermediate UMD spaces enters into Theorem~\ref{thm:main}.

\begin{proposition}
If $X=X_\theta=[X_0,X_1]_\theta$ is a complex interpolation space between a UMD space $X_0$ and a Hilbert space $X_1$, then $X$ has Fourier tile-type $q=2/\theta$.
\end{proposition}


\begin{proof}
By definition, we need to show that $X$ has Fourier tile-type $(q,\alpha)$ for every $\alpha\in(0,1)$. For the rest of the proof, fix one such $\alpha$.

We use Lemma~\ref{lem:upTreesOnly} to restrict to collections $\mathscr{T}$ of up-trees only, and reformulate the left side of the Fourier tile-type estimate as
\begin{equation*}
  \Norm{A_{\mathscr{T}}f}{\ell^q(\mathscr{T};L^q(\R;X))},\qquad
  A_{\mathscr{T}}f:=\Big\{\operatorname{Mod}_{-c(\omega_{\mathbb{T}_u})}\sum_{P\in\mathbb{T}}\phi_P \pair{f}{\phi_P}\Big\}_{\mathbb{T}\in\mathscr{T}}.
\end{equation*}
Note that the modulation $\operatorname{Mod}_{-c(\omega_{\mathbb{T}_u})}$ does not affect the value of the $L^q(\R;X)$-norm. However,
\begin{equation*}
  f\mapsto \sum_{P\in\mathbb{T}}\operatorname{Mod}_{-c(\omega_{\mathbb{T}_u})}\phi_P \pair{f}{\operatorname{Mod}_{-c(\omega_{\mathbb{T}_u})}\phi_P}
\end{equation*}
is a Calder\'on--Zygmund operator, and therefore it maps $L^\infty(\R;X_0)$ to $\BMO(\R;X_0)$ for any UMD space $X_0$. Moreover, the Calder\'on--Zygmund norms are uniform over the choice of different up-trees, and hence
\begin{equation*}
  \Norm{A_{\mathscr{T}}f}{\ell^\infty(\mathscr{T};\BMO(\R;X_0))}
  \lesssim \Norm{\{\operatorname{Mod}_{c(\omega_{\mathbb{T}_u})}f\}_{\mathbb{T}\in\mathscr{T}}}{\ell^\infty(\mathscr{T};L^\infty(\R;X_0))}
  =\Norm{f}{L^\infty(\R;X_0)}.
\end{equation*}
So we have
\begin{equation}\label{eq:inftyEnd}
  A_{\mathscr{T}}:L^\infty(\R;X_0)\to \ell^\infty(\mathscr{T};\BMO(\R;X_0)),
\end{equation}
which provides one end-point for interpolation.

Another one is obtained from Proposition~\ref{prop:HilbertTileType}. Abbreviating
\begin{equation*}
  S:=\sum_{\mathbb{T}\in\mathscr{T}}\abs{I_{\mathbb{T}}},
\end{equation*}
it tells that
\begin{equation*}
\begin{split}
  &\Norm{A_{\mathscr{T}}f}{\ell^2(\mathscr{T};L^2(\R;X_1))} \\
  &\lesssim\Norm{f}{L^2(\R;X_1)}+(\sqrt{S}\Norm{f}{L^\infty(\R;X_1)})^{1-\alpha}\Norm{f}{L^2(\R;X_1)}^\alpha \\
  &=\Norm{f}{L^2(\R;X_1)}+(\eps^{\alpha}\sqrt{S}\Norm{f}{L^\infty(\R;X_1)})^{1-\alpha}(\eps^{-(1-\alpha)}\Norm{f}{L^2(\R;X_1)})^\alpha \\
  &\leq\Norm{f}{L^2(\R;X_1)}+\eps^{\alpha}\sqrt{S}\Norm{f}{L^\infty(\R;X_1)}+\eps^{-(1-\alpha)}\Norm{f}{L^2(\R;X_1)}
\end{split}
\end{equation*}
for any $\eps>0$. This says that we have the boundedness
\begin{equation}\label{eq:2End}
  A_{\mathscr{T}}:L^2(\R;X_1)\cap L^\infty(\R;X_1)\to \ell^2(\mathscr{T};L^2(\R;X_1)),
\end{equation}
where $L^2(\R;X_1)\cap L^\infty(\R;X_1)$ is equipped with the norm
\begin{equation*}
  \Norm{f}{L^2(\R;X_1)\cap L^\infty(\R;X_1)}:=\eps^{\alpha}\sqrt{S}\Norm{f}{L^\infty(\R:X_1)}+(1+\eps^{\alpha-1})\Norm{f}{L^2(\R:X_1)}.
\end{equation*}

Interpolating between \eqref{eq:inftyEnd} and \eqref{eq:2End} by the complex method $[\ ,\ ]_\theta$, and using Lemma~\ref{lem:interpolation} and standard results to identify the resulting interpolation spaces in the domain and the range, we deduce that
\begin{equation*}
  A_{\mathscr{T}}:L^q(\R;X_\theta)\cap L^\infty(\R;X_\theta)\to \ell^q(\mathscr{T};L^q(\R;X_\theta)),
\end{equation*}
where $L^q(\R;X_\theta)\cap L^\infty(\R;X_\theta)$ is equipped with the norm
\begin{equation*}
  \Norm{f}{L^q(\R;X_\theta)\cap L^\infty(\R;X_\theta)}:=(\eps^{\alpha}\sqrt{S})^{\theta}\Norm{f}{L^\infty(\R:X_\theta)}+(1+\eps^{\alpha-1})^\theta\Norm{f}{L^q(\R:X_\theta)}.
\end{equation*}
In other words, we have
\begin{equation*}
  \Norm{A_{\mathscr{T}}f}{\ell^q(\mathscr{T};L^q(\R;X_\theta))}
  \lesssim\Norm{f}{L^q(\R;X_\theta)}+(\eps^{\alpha}\sqrt{S})^{\theta}\Norm{f}{L^\infty(\R:X_\theta)}+\eps^{(\alpha-1)\theta}\Norm{f}{L^q(\R:X_\theta)},
\end{equation*}
and this bound holds for every $\eps>0$ with the implied constant independent of~$\eps$.

We choose $\eps$ so as to equate the last two terms, which gives
\begin{equation*}
   \eps^\theta=\frac{\Norm{f}{q}}{\sqrt{S}^\theta\Norm{f}{\infty}}.
\end{equation*}
Substituting this leads to the final bound
\begin{equation*}
  \Norm{A_{\mathscr{T}}f}{\ell^q(\mathscr{T};L^q(\R;X_\theta))}
  \lesssim \Norm{f}{L^q(\R;X_\theta)}+\sqrt{S}^{(1-\alpha)\theta}\Norm{f}{L^\infty(\R:X_\theta)}^{1-\alpha}\Norm{f}{L^q(\R:X_\theta)}^\alpha.
\end{equation*}
Recalling that
\begin{equation*}
  \sqrt{S}^\theta=\sqrt{\sum_{\mathbb{T}\in\mathscr{T}}\abs{I_{\mathbb{T}}}}^{2/q}=\Big(\sum_{\mathbb{T}\in\mathscr{T}}\abs{I_{\mathbb{T}}}\Big)^{1/q},
\end{equation*}
we recognize the last bound as precisely the Fourier tile-type $(q,\alpha)$ estimate for $X_\theta$.
\end{proof}


\section{Carleson's theorem for large $p$}

The proof of Carleson's theorem \`a la Lacey--Thiele \cite{LT:MRL} is based on controlling two key quantities associated to any collection of tiles $\mathbb{P}$: density and energy. (The terminology is slightly variable over different papers on the subject.) More precisely, it is shown that a special Carleson operator, where the sum is over a tree of tiles, is controlled by a product of the density and the energy of the tree in question (the Tree lemma), whereas any collection of tiles---as in the summation condition for the general Carleson operator---can be recursively divided into trees in such a way that the density and energy are reduced at each step of the iteration (the Density and Energy lemmas). In this section, we adapt these ideas to prove the vector-valued Carleson theorem for large exponents $p$; the general case will require a further elaboration of the argument, which we postpone to the last two sections.

Recall that we want to prove the estimate
\begin{equation*}
  \abs{\pair{Cf}{g}}\lesssim\abs{F}^{1/p}\abs{E}^{1/p'}
\end{equation*}
for all $f\in L^\infty(\R;X)$ and $g\in L^\infty(\R;X^*)$ with $\abs{f}\leq 1_F$ and $\abs{g}\leq 1_E$. Henceforth, we will consider the functions $f,g$ and the measurable sets $F,E$ fixed; the density and energy will depend on them, but we will not indicate it explicitly.

The density of a collection of tiles $\mathbb{P}$ is defined exactly as in the scalar case:
\begin{equation*}
  \operatorname{density}(\mathbb{P}):=\sup_{P\in\mathbb{P}}\sup_{P'\geq P}\int_{E_{P'}}v_{I_{P'}},\qquad
  v_I(x):=\frac{1}{\abs{I}}\Big(1+\frac{\abs{x-c(I)}}{\abs{I}}\Big)^{-10},
\end{equation*}
where $E_{P'}:=E\cap\{x:N(x)\in\omega_{P'}\}$, and $x\mapsto N(x)$ is the arbitrary but fixed measurable functions in the definition of the Carleson operator. In the inner supremum, we go through all tiles $P'$ such that $P'\geq P$.
A trivial bound is $\operatorname{density}(\mathbb{P})\lesssim 1$ for any collection $\mathbb{P}$.

Our definition of energy involves an exponent $q$, which we take to be a Fourier tile-type exponent for the underlying space $X$. In the scalar case, the classical choice is $q=2$, and orthogonality leads to a different but equivalent formulation of the quantity below. In our general setting, the definition of energy is as follows:
\begin{equation*}
  \operatorname{energy}(\mathbb{P})
  :=\sup_{\mathbb{T}\subseteq\mathbb{P}}\Delta(\mathbb{T}),\qquad
  \Delta(\mathbb{T}):=\Big(\frac{1}{\abs{I_{\mathbb{T}}}}\int\Babs{\sum_{P\in\mathbb{T}_u}\pair{f}{\phi_P}\phi_P}^q\Big)^{1/q}.
\end{equation*}
Recall that we apply this to $f\in L^\infty(\R;X)$ with $\abs{f}\leq 1_F$. Then there is a universal upper bound $\operatorname{energy}(\mathbb{P})\lesssim 1$ valid for any collection $\mathbb{P}$. We omit the (reasonably standard) proof at this point, since a more general upper bound is established in Corollary~\ref{cor:impEnergy}.

%
%

We turn to the three key lemmas.

\begin{proposition}[Tree lemma]
If $\mathbb{T}$ is a tree, then
\begin{equation*}
  \sum_{P\in\mathbb{T}}
  \abs{\pair{f}{\phi_P}\pair{\phi_P}{ g 1_{\{N(\cdot)\in \omega_{P_u}\}}}}
  \lesssim\operatorname{density}(\mathbb{T})\operatorname{energy}(\mathbb{T})\abs{I_{\mathbb{T}}},
\end{equation*}
for $f\in L^\infty(\R;X)$ and $g\in L^\infty(\R;X^*)$ with $\abs{f}\leq 1_F$ and $\abs{g}\leq 1_E$.
\end{proposition}


\begin{proof}
This essentially repeats \cite[Sec.~6]{LT:MRL}, and we only give a sketch.

Let $\mathscr{J}$ be the collection of maximal dyadic intervals $J$ with the property that $2J\not\supseteq I_P$ for any $P\in\mathbb{T}$. Then $\mathscr{J}$ is a partition of $\R$ such that every $J\in\mathscr{J}$ satisfies either $J\subseteq 5I_{\mathbb{T}}$, or $\abs{J}\geq 2\abs{I_{\mathbb{T}}}$ and $\tfrac12\abs{J}\leq\dist(J,I_{\mathbb{T}})\leq 2\abs{J}$. It follows that the left side is dominated by
\begin{equation*}
    \sum_{J\in\mathscr{J}}\sum_{\substack{P\in\mathbb{T}\\ \abs{I_P}\leq\abs{J}}}\abs{\pair{f}{\phi_P}}\Norm{\phi_P1_{E_{P_{u}}}}{L^1(J)}
       +\sum_{\substack{J\in\mathscr{J}\\ J\subseteq 5I_{\mathbb{T}}}}\BNorm{\sum_{\substack{P\in\mathbb{T}\\ \abs{I_P}>\abs{J}}}
         \epsilon_P\pair{f}{\phi_P}\phi_P 1_{E_{P_{u}}}}{L^1(J)}
\end{equation*}
for some sings $\epsilon_P$ taking care of the ratio of $\pair{f}{\phi_P}\pair{\phi_P}{ g 1_{\{N(\cdot)\in \omega_{P_u}\}}}$ and its absolute value.
The first part only involves the size of individual coefficients
\begin{equation}\label{eq:naiveEnergy}
  \abs{\pair{f}{\phi_P}}\leq\abs{I_P}^{1/2}\operatorname{energy}(\mathbb{T}), 
\end{equation}
and is estimated verbatim to the scalar case \cite[Sec.~6]{LT:MRL}.

In the second part, the support of the function inside the $L^1(J)$ norm is contained in a set $G_J$ with $\abs{G_J}\lesssim\operatorname{density}(\mathbb{T})\abs{J}$---an algebraic property inherited from the scalar case---, and we estimate $\Norm{F_J}{L^1(J)}\leq\abs{G_J}\Norm{F_J}{L^\infty(J)}$. Splitting $\mathbb{T}=\mathbb{T}_d\cup\mathbb{T}_u$ into a down-tree and an up-tree, the part $\mathbb{T}_d$ also needs the na\"ive bound \eqref{eq:naiveEnergy} only, still verbatim to \cite[Sec.~6]{LT:MRL}.

In the final case, $P\in\mathbb{T}_u$, the sets $E_{P_u}$ are nested, and we have
\begin{equation*}
\begin{split}
  F_J(x): &=\sum_{\substack{P\in\mathbb{T}_u\\ \abs{I_P}>\abs{J}}}
  \epsilon_P\pair{f}{\phi_P}\phi_P(x) 1_{E_{P_{u}}}(x)
  =\sum_{\substack{P\in\mathbb{T}_{u}\\ \abs{J}<\abs{I_P}\leq\abs{I_x}}}\epsilon_P\pair{f}{\phi_P}\phi_P(x) \\
  &=\operatorname{Mod}_{c(\omega_{\mathbb{T}_u})}\Big((D^1_{\abs{J}}\chi-D^1_{\abs{I_x}}\chi)*
     \operatorname{Mod}_{-c(\omega_{\mathbb{T}_u})}\sum_{P\in\mathbb{T}_u}
     \epsilon_P\pair{f}{\phi_P}\phi_P\Big)(x),
\end{split}
\end{equation*}
for a suitable length $\abs{I_x}$ and a frequency-localized $\chi\in\mathscr{S}(\R)$. Hence
\begin{equation*}
  \Norm{F_J}{L^\infty(J;X)}
  \lesssim\inf_J M\Big(\sum_{P\in\mathbb{T}_u}
     \epsilon_P\pair{f}{\phi_P}\phi_P\Big),
\end{equation*}
and finally, by the disjointness of $J\in\mathscr{J}$, H\"older's inequality, the maximal inequality, Proposition~\ref{prop:czo}(3) and the definition of energy, we have
\begin{equation*}
\begin{split}
  \sum_{\substack{J\in\mathscr{J}\\ J\subseteq 5I_{\mathbb{T}}}} &\Norm{F_J}{L^1(J;X)} 
  \lesssim \sum_{\substack{J\in\mathscr{J}\\ J\subseteq 5I_{\mathbb{T}}}} \operatorname{density}(\mathbb{T})\abs{J}\times \inf_J
     M\Big(\epsilon_P\sum_{P\in\mathbb{T}_{u}}\pair{f}{\phi_P}\phi_P\Big) \\
    &\leq\operatorname{density}(\mathbb{T})\int_{5I_{\mathbb{T}}}M\Big(\sum_{P\in\mathbb{T}_{u}}
        \epsilon_P\pair{f}{\phi_P}\phi_P\Big)(x)\ud x \\
   &\lesssim\operatorname{density}(\mathbb{T})\abs{I_{\mathbb{T}}}^{1/q'}\BNorm{\sum_{P\in\mathbb{T}_{u}}
       \epsilon_P\pair{f}{\phi_P}\phi_P}{L^q(\R;X)} \\
   &\lesssim\operatorname{density}(\mathbb{T})\abs{I_{\mathbb{T}}}^{1/q'}\BNorm{\sum_{P\in\mathbb{T}_{u}}
       \pair{f}{\phi_P}\phi_P}{L^q(\R;X)} \\
   &\leq\operatorname{density}(\mathbb{T})\abs{I_{\mathbb{T}}}^{1/q'}\operatorname{energy}(\mathbb{T})\abs{I_{\mathbb{T}}}^{1/q}.\qedhere
\end{split}
\end{equation*}
\end{proof}


\begin{proposition}[Density lemma]
If $\mathbb{P}$ is a finite collection of tiles, then we have a splitting
\begin{equation*}
  \mathbb{P}=\mathbb{P}_{\operatorname{sparse}}\cup\bigcup_j\mathbb{T}_j,
\end{equation*}
where $\operatorname{density}(\mathbb{P}_{\operatorname{sparse}})\leq 2^{-1}\operatorname{density}(\mathbb{P})$ and
\begin{equation*}
  \sum_j \abs{I_{\mathbb{T}_j}}\lesssim\operatorname{density}(\mathbb{P})^{-1}\abs{E}.
\end{equation*}
\end{proposition}

\begin{proof}
This purely scalar result is \cite[Proposition 3.1]{LT:MRL}.
\end{proof}

The Fourier tile-type makes its one and only appearance in the next result:

\begin{proposition}[Energy lemma]
Let $X$ have Fourier tile-type $(q,\alpha)$.
If $\mathbb{P}$ is a finite collection of tiles, then we have a splitting
\begin{equation*}
  \mathbb{P}=\mathbb{P}_{\operatorname{small}}\cup\bigcup_j\mathbb{T}_j,
\end{equation*}
where $\operatorname{energy}(\mathbb{P}_{\operatorname{small}})\leq 2^{-1}\operatorname{energy}(\mathbb{P})$ and
\begin{equation*}
  \sum_j \abs{I_{\mathbb{T}_j}}\lesssim\operatorname{energy}(\mathbb{P})^{-q/\alpha}\abs{F}.
\end{equation*}
\end{proposition}

\begin{proof}
Among all maximal trees $\mathbb{T}\subseteq\mathbb{P}$ with $\Delta(\mathbb{T})>\operatorname{energy}(\mathbb{T})/2$, let $\mathbb{T}_1$ be one with the minimal $c(\omega_{\mathbb{T}})$. Remove $\mathbb{T}_1$ from $\mathbb{P}$ and iterate this selection as long as possible. What is left of $\mathbb{P}$ after these removals qualifies for $\mathbb{P}_{\operatorname{small}}$. Let $\mathscr{T}$ be the collection of the up-trees $\mathbb{T}_{j,u}$ corresponding to the chosen trees $\mathbb{T}_j$. By construction,
\begin{equation*}
  \sum_j\abs{I_{\mathbb{T}_j}}
  \lesssim\operatorname{energy}(\mathbb{P})^{-q}\sum_{\mathbb{T}\in\mathscr{T}}\BNorm{\sum_{P\in\mathbb{T}}\pair{f}{\phi_P}\phi_P}{L^q(\R;X)}^q.
\end{equation*}
From the construction one checks that $\mathscr{T}$ satisfies \eqref{eq:disjProp}, so by definition of Fourier tile-type we have
\begin{equation*}
\begin{split}
   \sum_j\abs{I_{\mathbb{T}_j}}
  &\lesssim\operatorname{energy}(\mathbb{P})^{-q}\Big(\Norm{f}{L^q(\R_+;X)}^q+
      \Big[\sum_j\abs{I_{\mathbb{T}_j}}\Big]^{1-\alpha}\Norm{f}{L^\infty(\R_+;X)}^{(1-\alpha)q}\Norm{f}{L^q(\R;X)}^{\alpha q}\Big) \\
  &\lesssim\operatorname{energy}(\mathbb{P})^{-q}\Big(\abs{F}+
      \Big[\sum_j\abs{I_{\mathbb{T}_j}}\Big]^{1-\alpha}\abs{F}^\alpha\Big)
\end{split}
\end{equation*}
If the first term dominates, we get
\begin{equation*}
  \sum_j\abs{I_{\mathbb{T}_j}}\lesssim\operatorname{energy}(\mathbb{P})^{-q}\abs{F},
\end{equation*}
and if the second term dominates,
\begin{equation*}
  \sum_j\abs{I_{\mathbb{T}_j}}\lesssim \operatorname{energy}(\mathbb{P})^{-q/\alpha}\abs{F}.
\end{equation*}
Since $\operatorname{energy}(\mathbb{P})\lesssim 1$, the second bound is always the larger one.
\end{proof}

Iterating the density and energy lemmas in tandem, we obtain the following final form of the decomposition.

\begin{lemma}
If $\mathbb{P}$ is a finite collection of tiles, then we have a splitting
\begin{equation*}
  \mathbb{P}=\bigcup_{n\in\Z}\bigcup_j\mathbb{T}_{n,j}\cup \mathbb{P}_\infty,
\end{equation*}
where $\operatorname{density}(\mathbb{T}_{n,j})\leq\abs{E}2^{-n}$, $\operatorname{energy}(\mathbb{T}_{n,j})\leq\abs{F}^{\alpha/q}2^{-n\alpha/q}$,
\begin{equation*}
  \sum_j \abs{I_{\mathbb{T}_{n,j}}}\lesssim 2^n,
\end{equation*}
and $\operatorname{density}(\mathbb{P}_{\infty})=\operatorname{energy}(\mathbb{P}_{\infty})=0$.
\end{lemma}

We are ready to complete the proof of Theorem~\ref{thm:main} for large values of $p$:

\begin{proposition}\label{prop:largePCarleson}
Let $X$ have Fourier tile-type $(q,\alpha)$.
Then the Carleson operator is bounded from $L^{p,1}(\R;X)$ to $L^{p,\infty}(\R;X)$ for $p\in[q/\alpha)$; hence, by interpolation, on $L^{p}(\R;X)$ for every $p\in(q/\alpha,\infty)$. In particular, if $X$ has Fourier tile-type $q$, then the Carleson operator is bounded on $L^p(\R;X)$ for every $p\in(q,\infty)$.
\end{proposition}

\begin{proof}
Let $f\in L^\infty(\R;X)$ and $g\in L^\infty(\R;X^*)$ satisfy $\abs{f}\leq 1_F$, $\abs{g}\leq 1_E$. Then
\begin{equation*}
\begin{split}
  \sum_{P\in\mathbb{P}}
  &\abs{\pair{f}{\phi_P}\pair{\phi_P}{ g 1_{\{N(\cdot)\in \omega_u\}}}} \\
  &\leq \sum_n\sum_j \sum_{P\in\mathbb{T}_{n,j}}
  \abs{\pair{f}{\phi_P}\pair{\phi_P}{ g 1_{\{N(\cdot)\in \omega_u\}}}} \\
  &\lesssim \sum_n\sum_j 
  \operatorname{density}(\mathbb{T}_{n,j})\operatorname{energy}(\mathbb{T}_{n,j})\abs{I_{\mathbb{T}_{n,j}}} \\
  &\lesssim \sum_n\min\{1,\abs{E}2^{-n}\}\min\{1,\abs{F}^{\alpha/q}2^{-n\alpha/q}\}2^n
\end{split}
\end{equation*}

\subsubsection*{Case $\abs{E}\leq\abs{F}$:} Then
\begin{equation*}
\begin{split}
  \sum_n &\min\{1,\abs{E}2^{-n}\}\min\{1,\abs{F}^{\alpha/q}2^{-n\alpha/q}\}2^n \\
  &\leq\sum_{2^n\leq\abs{E}}2^n+\sum_{\abs{E}\leq 2^n\leq\abs{F}}\abs{E}+\sum_{2^n\geq\abs{F}}\abs{E}\abs{F}^{\alpha/q}2^{-n\alpha/q} \\
  &\lesssim \abs{E}\Big(1+\log\frac{\abs{F}}{\abs{E}}\Big)\lesssim\abs{F}^{1/p}\abs{E}^{1/p'}
\end{split}
\end{equation*}
for any $p\in(1,\infty)$.

\subsubsection*{Case $\abs{E}>\abs{F}$:} Then
\begin{equation*}
\begin{split}
  \sum_n &\min\{1,\abs{E}2^{-n}\}\min\{1,\abs{F}^{\alpha/q}2^{-n\alpha/q}\}2^n \\
  &\leq\sum_{2^n\leq\abs{F}}2^n+\sum_{\abs{F}\leq 2^n\leq\abs{E}}\abs{F}^{\alpha/q}2^{n(1-\alpha/q)}
       +\sum_{2^n\geq\abs{E}}\abs{E}\abs{F}^{\alpha/q}2^{-n\alpha/q} \\
  &\lesssim \abs{F}^{\alpha/q}\abs{E}^{1-\alpha/q}
     \lesssim \abs{F}^{1/p}\abs{E}^{1/p'}
\end{split}
\end{equation*}
if and only if $p\geq q/\alpha$. Thus
\begin{equation*}
   \sum_{P\in\mathbb{P}}
  \abs{\pair{f}{\phi_P}\pair{\phi_P}{ g 1_{\{N(\cdot)\in \omega_u\}}}}
  \lesssim\Norm{f}{L^{p,1}(\R;X)}\Norm{g}{L^{p',1}(\R;X^*)}
\end{equation*}
for all $p\in[q/\alpha,\infty)$.
\end{proof}

\section{Improved energy estimates}

In order to extend the proof of Carleson's theorem to all $p>1$, we need an improvement over the universal energy  bound $\operatorname{energy}(\mathbb{P})\lesssim 1$ for $\abs{f}\leq 1_F$. This asks for a development of vector-valued analogues of some results on wavelets and square-functions appearing in \cite[Chapter 2]{Thiele:book}. We start with a boundedness result for a certain mixture of continuous and dyadic Calder\'on--Zygmund operators.

\begin{lemma}\label{lem:mixedOperator}
For every $I\in\mathscr{D}$, let $\psi_I$ be a smooth $L^2$-bump on $I$, i.e.,
\begin{equation*}
  \abs{\psi_I(x)}+ \abs{I}\cdot\abs{\psi_I'(x)}\lesssim\frac{1}{\abs{I}^{1/2}}\Big(1+\frac{\abs{x-c(I)}}{\abs{I}}\Big)^{-20},
\end{equation*}
with vanishing integral. If $X$ is a UMD space and $\mathscr{I}$ is any finite collection of dyadic intervals, then
\begin{equation*}
  f\mapsto \sum_{I\in\mathscr{I}}\pair{f}{\psi_I}h_I
\end{equation*}
is bounded on $L^p(\R;X)$ for $p\in(1,\infty)$ and from $L^1(\R;X)$ to $L^{1,\infty}(\R;X)$, uniformly over the choice of $\mathscr{I}$.
\end{lemma}

\begin{proof}
Let $\tilde\psi_I$, $I\in\mathscr{D}$, be regular orthonormal wavelets, i.e., just like the $\psi_I$, and in addition pairwise orthogonal. The may be taken to be the translations and dilations of a single function $\tilde\psi_{[0,1)}$, which is not identically vanishing on $[0,1)$. Then it follows that $\fint_I\abs{\tilde\psi_I(x)}\ud x\gtrsim\abs{I}^{-1/2}$. By the unconditionality of the Haar functions, the contraction principle and Stein's inequality,
\begin{equation*}
\begin{split}
  &\BNorm{\sum_{I\in\mathscr{I}}\pair{f}{\psi_I}h_I}{L^p(\R;X)}
  \lesssim\BNorm{\sum_{I\in\mathscr{I}}\eps_I\pair{f}{\psi_I}h_I}{L^p(\R\times\Omega;X)} \\
  &\qquad\lesssim\BNorm{\sum_{I\in\mathscr{I}}\eps_I\pair{f}{\psi_I}\frac{1_I}{\abs{I}^{1/2}}}{L^p(\R\times\Omega;X)} 
  \lesssim\BNorm{\sum_{I\in\mathscr{I}}\eps_I\pair{f}{\psi_I}\Exp_I\abs{\tilde\psi_I}}{L^p(\R\times\Omega;X)} \\
  &\qquad\lesssim\BNorm{\sum_{I\in\mathscr{I}}\eps_I\pair{f}{\psi_I}\tilde\psi_I}{L^p(\R\times\Omega;X)}.
\end{split}
\end{equation*}
The operator $f\mapsto \sum_{I\in\mathscr{I}}\eps_I\pair{f}{\psi_I}\tilde\psi_I$ has a standard Calder\'on--Zygmund kernel, uniformly in $\mathscr{I}$ and the choice of the signs $\eps_I$. It is also bounded on the scalar $L^2(\R)$ space, since
\begin{equation*}
  \BNorm{\sum_{I\in\mathscr{I}}\eps_I\pair{f}{\psi_I}\tilde\psi_I}{2}^2
  =\sum_{I\in\mathscr{I}}\abs{\pair{f}{\psi_I}}^2\lesssim\Norm{f}{2}^2,
\end{equation*}
where the first step follows from the orthonormality of the wavelets $\tilde\psi_I$, and the second is \cite[Proposition 2.3.1]{Thiele:book}.
Hence, by Figiel's $T(1)$ theorem in UMD spaces \cite{Figiel}, it is bounded on all $L^p(\R;X)$. By the above chain of estimates, so is the original operator.

For the boundedness from $L^1(\R;X)$ to $L^{1,\infty}(\R;X)$, it is enough to observe that the original operator has kernel $\sum_{I\in\mathscr{I}}h_I(x)\psi_I(y)$, which satisfies the standard Calder\'on--Zygmund size estimate as well as the standard regularity estimate in the $y$-variable. This is enough to deduce the the weak $L^1$ estimate from the $L^p$ boundedness already established.
\end{proof}

\begin{lemma}
For every $I\in\mathscr{D}$, let $\psi_I$ be a smooth $L^2$-bump on $I$ with vanishing integral. If $X$ is a UMD space and $\mathscr{I}$ is any finite collection of dyadic intervals, then
\begin{equation*}
  \BNorm{\sum_{I\in\mathscr{I}}\pair{f}{\psi_I}\psi_I}{L^p(\R;X)}
  \lesssim\BNorm{\sum_{I\in\mathscr{I}}\pair{f}{\psi_I}h_I}{L^p(\R;X)}
\end{equation*}
for $p\in(1,\infty)$, uniformly over the choice of $\mathscr{I}$.
\end{lemma}

\begin{proof}
We argue by duality. Namely, for some $g\in L^{p'}(\R;X^*)$ of norm one,
\begin{equation*}
\begin{split}
  \BNorm{\sum_{I\in\mathscr{I}}\pair{f}{\psi_I}\psi_I}{L^p(\R;X)}
  &\lesssim\sum_{I\in\mathscr{I}}\pair{f}{\psi_I}\pair{\psi_I}{g} \\
  &=\Bpair{\sum_{I\in\mathscr{I}}\pair{f}{\psi_I}h_I}{\sum_{I'\in\mathscr{I}}\pair{g}{\psi_{I'}}h_{I'}} \\
  &\leq\BNorm{\sum_{I\in\mathscr{I}}\pair{f}{\psi_I}h_I}{L^p(\R;X)}\BNorm{\sum_{I'\in\mathscr{I}}\pair{g}{\psi_{I'}}h_{I'}}{L^{p'}(\R;X^*)},
\end{split}
\end{equation*}
and the second factor is bounded by $\Norm{g}{L^{p'}(\R;X^*)}\leq 1$ by the previous lemma.
\end{proof}

The following result is an improved square function estimate, a vector-valued extension of \cite[Proposition 2.4.1]{Thiele:book}:

\begin{proposition}
Let $X$ be a UMD space and $f\in L^1_{\loc}(\R)$.
Let $\psi_I$ be smooth $L^2$-bumps with vanishing integral, and $\mathscr{I}\subseteq\{I\in\mathscr{D}:\inf_I Mf\leq\lambda\}$ be a finite collection of dyadic intervals. Then for all $p\in(1,\infty)$,
\begin{equation*}
  \BNorm{\sum_{\substack{I\in\mathscr{I}\\ I\subseteq K}}\pair{f}{\psi_I}h_I}{L^p(\R;X)}\lesssim \lambda\abs{K}^{1/p}.
\end{equation*}
\end{proposition}

\begin{proof}
Note that $\inf_I Mf\leq\lambda$ implies that $\pair{f}{\psi_I}$ is well-defined. We denote
\begin{equation*}
  \tilde{f}:=\sum_{I\in\mathscr{I}}\pair{f}{\psi_I}h_I.
\end{equation*}
Then, denoting by $\mathscr{I}^*(K)$ the maximal elements $I\in\mathscr{I}$ with $I\subseteq K$,
\begin{equation*}
 1_K(\tilde{f}-\ave{\tilde{f}}_K)= \sum_{\substack{I\in\mathscr{I}\\I\subseteq K}}\pair{f}{\psi_I}h_I
  = \sum_{J\in\mathscr{I}^*(K)}\sum_{\substack{I\in\mathscr{I}\\ I\subseteq J} }\pair{f}{\psi_I}h_I.
\end{equation*}

For each $J\in\mathscr{I}^*(K)$, we apply Lemma~\ref{lem:mixedOperator} to the operator
\begin{equation*}
   f\mapsto \sum_{\substack{I\in\mathscr{I}\\ I\subseteq J} }\pair{f}{\psi_I}h_I
\end{equation*}
to deduce its boundedness from $L^1(\R;X)$ into $L^{1,\infty}(\R;X)$. Hence
\begin{equation}\label{eq:JNin}
   \BNorm{\sum_{\substack{I\in\mathscr{I}\\ I\subseteq J} }\pair{1_{2J}f}{\psi_I}h_I}{L^{1,\infty}(\R;X)}
   \lesssim\Norm{1_{2J}f}{L^1(\R;X)}
   \lesssim\abs{J}\inf_J Mf\leq\lambda\abs{J}.
\end{equation}
On the other hand, for $I\subseteq J$ we have
\begin{equation*}
\begin{split}
  \abs{\pair{1_{(2J)^c}f}{\psi_I}}
  &\lesssim\int_{(2J)^c}\frac{\abs{f(x)}}{\abs{I}^{1/2}}\Big(1+\frac{\abs{x-c(I)}}{\abs{I}}\Big)^{-20}\ud x \\
  &\lesssim\abs{I}^{1/2}\Big(\frac{\abs{I}}{\abs{J}}\Big)^{10}\int_{\R}\frac{\abs{f(x)}}{\abs{I}}\Big(1+\frac{\abs{x-c(I)}}{\abs{I}}\Big)^{-10}\ud x \\
  &\lesssim\abs{I}^{1/2}\Big(\frac{\abs{I}}{\abs{J}}\Big)^{9}\int_{\R}\frac{\abs{f(x)}}{\abs{J}}\Big(1+\frac{\abs{x-c(J)}}{\abs{J}}\Big)^{-10}\ud x \\
  &\lesssim\abs{I}^{1/2}\Big(\frac{\abs{I}}{\abs{J}}\Big)^{9}\inf_J Mf \leq\abs{I}^{1/2}\Big(\frac{\abs{I}}{\abs{J}}\Big)^{9}\lambda,
\end{split}
\end{equation*}
hence
\begin{equation*}
\begin{split}
  \Babs{\sum_{\substack{I\in\mathscr{I}\\ I\subseteq J} }\pair{f}{\psi_I}h_I}
  &\lesssim\sum_{\substack{I\in\mathscr{D}\\ I\subseteq J} }\Big(\frac{\abs{I}}{\abs{J}}\Big)^{9}\lambda\cdot 1_I \\
  &=\sum_{k=0}^\infty \sum_{\substack{I\in\mathscr{D}, I\subseteq J \\ \abs{I}=2^{-k}\abs{J} }}2^{-9k}\lambda\cdot 1_I 
  \lesssim \lambda\cdot 1_J,
\end{split}
\end{equation*}
and thus also
\begin{equation}\label{eq:JNout}
   \BNorm{\sum_{\substack{I\in\mathscr{I}\\ I\subseteq J} }\pair{1_{(2J)^c}f}{\psi_I}h_I}{L^{1,\infty}(\R;X)}
   \lesssim \lambda\abs{J}.
\end{equation}
A combination of \eqref{eq:JNin} and \eqref{eq:JNout} shows that
\begin{equation*}
  \BNorm{\sum_{\substack{I\in\mathscr{I}\\ I\subseteq J} }\pair{f}{\psi_I}h_I}{L^{1,\infty}(\R;X)}
  \lesssim \lambda\abs{J}.
\end{equation*}
Since the terms corresponding to different $J$ are disjointly supported, we get the first estimate in
\begin{equation*}
\begin{split}
  \Norm{1_K(\tilde{f}-\ave{\tilde{f}}_K)}{L^{1,\infty}(\R;X)}
  &\leq\sum_{J\in\mathscr{I}^*(K)}\BNorm{\sum_{\substack{I\in\mathscr{I}\\ I\subseteq J} }\pair{f}{\psi_I}h_I}{L^{1,\infty}(\R;X)} \\
  &\lesssim \sum_{J\in\mathscr{I}^*(K)}\lambda\abs{J}
  \leq\lambda\abs{K}.
\end{split}
\end{equation*}
By the John--Str\"omberg inequality, we have $\Norm{\tilde{f}}{\BMO(\R;X)}\lesssim\lambda$, and then by the John--Nirenberg inequality that
\begin{equation*}
  \Norm{1_K(\tilde{f}-\ave{\tilde{f}}_K)}{L^p(\R;X)}\lesssim\lambda\abs{K}^{1/p}.\qedhere
\end{equation*}
\end{proof}

Now finally we are ready for the improved energy estimate:

\begin{corollary}\label{cor:impEnergy}
Let $\mathbb{P}$ be a finite collection of tiles such that $\inf_{I_P}Mf\leq\lambda$ for all $P\in\mathbb{P}$. Then
\begin{equation*}
  \operatorname{energy}(\mathbb{P})\lesssim\lambda.
\end{equation*}
\end{corollary}

\begin{proof}
Let $\mathbb{T}\subseteq\mathbb{P}$ be an up-tree. The tiles $P\in\mathbb{T}$ are in one-to-one correspondence with their time intervals $I_P$, and hence
\begin{equation*}
  \psi_{I_P}:=\operatorname{Mod}_{-c(\omega_{\mathbb{T}})}\phi_P,\qquad P\in\mathbb{T},
\end{equation*}
is well defined. These functions are smooth $L^2$-bumps with vanishing integral. Hence, by applying the previous results, we deduce that
\begin{equation*}
\begin{split}
  \BNorm{\sum_{P\in\mathbb{T}}\pair{f}{\phi_P}\phi_P}{L^q(\R;X)}
  &=\BNorm{\sum_{P\in\mathbb{T}}\pair{\operatorname{Mod}_{-c(\omega_{\mathbb{T}})}f}{\psi_{I_P}}\psi_{I_P}}{L^q(\R;X)} \\
  &\lesssim\BNorm{\sum_{P\in\mathbb{T}}\pair{\operatorname{Mod}_{-c(\omega_{\mathbb{T}})}f}{\psi_{I_P}}h_{I_P}}{L^q(\R;X)} 
    \lesssim\lambda\abs{I_{\mathbb{T}}}^{1/q},
\end{split}
\end{equation*}
since $\operatorname{Mod}_{-c(\omega_{\mathbb{T}})}f$ has the same maximal function as $f$.
\end{proof}

\section{Carleson's theorem for general $p>1$}

In this final section, we complete the proof of Theorem~\ref{thm:main} in full generality. Thus we want to prove that
\begin{equation*}
  \abs{\pair{Cf}{g}}\lesssim\Norm{f}{L^p(\R;X)}\Norm{g}{L^{p'}(\R;X)}
\end{equation*}
for all $p\in(1,\infty)$ and all functions making the right side finite. By well-known results concerning interpolation of generalized restricted weak-type inequalities, it suffices to show the following: For all sets $E,F\subseteq\R$ of finite measure, there exists a \emph{major subset} $\tilde{E}\subseteq E$ with $\abs{\tilde{E}}\geq\tfrac12\abs{E}$, such that for all $f\in L^\infty(F;X)$ and all $g\in L^\infty(\tilde{E};X^*)$ of norm one, we have
\begin{equation*}
  \abs{\pair{Cf}{g}}\lesssim\abs{F}^{1/p}\abs{E}^{1/p'}
\end{equation*}
for all $p\in(1,\infty)$. The proof splits into two cases according to the relative size of $E$ and $F$. (In the first case, we can simply take $\tilde{E}=E$.)

\begin{lemma}
Let $\abs{E}\leq\abs{F}$. Then
\begin{equation*}
  \abs{\pair{Cf}{g}}\lesssim \abs{E}\Big(1+\log\frac{\abs{F}}{\abs{E}}\Big)
\end{equation*}
for all $f\in L^\infty(F;X)$ and $g\in L^\infty(E;X^*)$ bounded by one.
\end{lemma}

\begin{proof}
This is contained in the proof of Proposition~\ref{prop:largePCarleson}.
\end{proof}

The other case is the more involved one; it is here that we need the improved energy estimate from the previous section.

\begin{lemma}
Let $\abs{E}>\abs{F}$. Then there exists $\tilde{E}\subseteq E$ with $\abs{\tilde{E}}\geq\tfrac12\abs{E}$ such that
\begin{equation*}
  \abs{\pair{Cf}{g}}\lesssim \abs{F}\Big(1+\log\frac{\abs{E}}{\abs{F}}\Big)
\end{equation*}
for all $f\in L^\infty(F;X)$ and $g\in L^\infty(\tilde{E};X^*)$ bounded by one.\end{lemma}

\begin{proof}
Let $G:=\{M(1_F)>K\abs{F}/\abs{E}\}$ and $\tilde{G}:=\{M(1_G)>\tfrac18\}$. Then $\abs{\tilde{G}}\leq\tfrac12\abs{E}$, and hence $\tilde{E}:=E\setminus\tilde{G}$ satisfies $\abs{\tilde{E}}\geq\tfrac12\abs{E}$. For $f$ and $g$ as in the assertion, we write
\begin{equation*}
  \sum_{P\in\mathbb{P}}\abs{\pair{f}{\phi_P}\pair{\phi_P}{g1_{E_{P_u}}}}
  =\sum_{\substack{P\in\mathbb{P}\\ 2I_P\not\subseteq \tilde{G}}}
  +\sum_{\substack{P\in\mathbb{P}\\ 2I_P\subseteq \tilde{G}}}
  =\sum_{\substack{P\in\mathbb{P}\\ 2I_P\not\subseteq \tilde{G}}}
  + \sum_{k=1}^\infty\sum_{\substack{P\in\mathbb{P}\\ 2^k I_P\subseteq \tilde{G} \\ 2^{k+1}I_P \not\subseteq \tilde{G}}}.
\end{equation*}

\subsubsection*{The part with $2I_P\subseteq \tilde{G}$}
For the $k$th term in this part, we observe the following estimates. First,
\begin{equation*}
\begin{split}
  \abs{\pair{f}{\phi_P}}
  &\lesssim\abs{I_P}^{1/2}\pair{\abs{f}}{v_{I_{P}}}
  \lesssim 2^k\abs{I_P}^{1/2}\pair{\abs{f}}{v_{2^{k+1}I_{P}}} \\
  &\lesssim 2^k\abs{I_P}^{1/2}\inf_{2^{k+1}I_P}Mf 
  \lesssim 2^k\abs{I_P}^{1/2}\frac{\abs{F}}{\abs{E}},
\end{split}
\end{equation*}
where $2^{k+1}I_P\not\subseteq\tilde{G}$ was used in the last step. Second, since
\begin{equation*}
  \supp g\subseteq\tilde{G}^c\subseteq (2^kI_P)^c,
\end{equation*}
we have
\begin{equation*}
\begin{split}
  \abs{\pair{\phi_P}{g1_{E_{P_u}}}}
  &\lesssim\abs{I_P}^{1/2}\int_{(2^k I_P)^c}\abs{g(x)}1_{E_{P_u}}(x)\frac{1}{\abs{I_P}}\Big(1+\frac{\abs{x-c(I_P)}}{\abs{I_P}}\Big)^{-20}\ud x \\
    &\lesssim\abs{I_P}^{1/2} 2^{-10k}\int 1_{\omega_P}(N(x))v_{I_P}(x)\ud x.
\end{split}
\end{equation*}
We also observe that the different intervals $I_P$ in this $k$th sum have overlap of at most $2$ at any point. In fact, it is easy to check that
\begin{equation*}
  I^{(k)}\subset 2^{k+1}I\subset 2^k I^{(2)},\qquad k\geq 1.
\end{equation*}
Thus, if $I$ satisfies $2^{k+1}I\not\subseteq\tilde{G}$, then no dyadic interval $J\subseteq I^{(2)}$ can satisfy $2^k J\subseteq\tilde{G}$. Hence,
\begin{equation*}
\begin{split}
  \sum_{\substack{P\in\mathbb{P}\\ 2^k I_P\subseteq \tilde{G} \\ 2^{k+1}I_P \not\subseteq \tilde{G}}}
    &\abs{\pair{f}{\phi_P}\pair{\phi_P}{g1_{E_{P_u}}}} \\
  &\leq \sum_{\substack{I: 2^k I\subseteq \tilde{G} \\ 2^{k+1}I \not\subseteq \tilde{G}}}
     \sum_{\omega:\abs{\omega}=1/\abs{I}} 2^k\abs{I}^{1/2}\frac{\abs{F}}{\abs{E}}\cdot 2^{-10k}\abs{I}^{1/2}\int 1_{\omega_u}(N(x))v_I(x)\ud x \\
  &\leq  2^{-9k}\frac{\abs{F}}{\abs{E}}\sum_{\substack{I: 2^k I\subseteq \tilde{G} \\ 2^{k+1}I \not\subseteq \tilde{G}}}\abs{I}
     \int v_I(x)\ud x  \\
  &\lesssim  2^{-9k}\frac{\abs{F}}{\abs{E}}\sum_{\substack{I: 2^k I\subseteq \tilde{G} \\ 2^{k+1}I \not\subseteq \tilde{G}}}\abs{I}
  \lesssim  2^{-9k}\frac{\abs{F}}{\abs{E}}\abs{\tilde{G}}
  \leq 2^{-9k}\abs{F},
\end{split}
\end{equation*}
where we used among other things the disjointness of dyadic $\omega$ of equal length, and the fact that all the $I$ in the sum have bounded overlap and are all contained in $\tilde{G}$. This is summable over $k=1,2,\ldots,$ and shows that
\begin{equation*}
  \sum_{\substack{P\in\mathbb{P}\\ 2I_P\subseteq \tilde{G}}}\abs{\pair{f}{\phi_P}\pair{\phi_P}{g1_{E_{P_u}}}}
  \lesssim\abs{F},
\end{equation*}
which is even a better bound than what we claimed for the full sum over $P\in\mathbb{P}$.

\subsubsection*{The part with $2I_P\not\subseteq \tilde{G}$}
The point is to check the conditions of Corollary~\ref{cor:impEnergy}, which will give us an improved energy estimate for the collection of tiles appearing in this sum. The condition that $2I_P\not\subseteq\tilde{G}$ means that there exists some $y\in 2I_P\subsetneq\tilde{G}$, which, by the definition of $\tilde{G}$ means that
\begin{equation*}
  \frac{\abs{2I_P\cap G}}{\abs{2I_P}}\leq M(1_G)(z)\leq\frac18,
\end{equation*}
and hence $\abs{I_P\cap G}\leq \tfrac18\abs{2I_P}=\tfrac14\abs{I_P}$. In particular, $I_P\not\subseteq G$, which by the definition of $G$ means the existence of some $z\in I_P\setminus G$ so that $M(1_F)\leq K\abs{F}/\abs{E}$. Since $\abs{f}\leq 1_F$, we finally obtain
\begin{equation*}
  \inf_{I_P}Mf\leq M(1_F)(z)\leq K\frac{\abs{F}}{\abs{E}}.
\end{equation*}
This means that the collection
\begin{equation*}
  \mathbb{P}':=\{P\in\mathbb{P}: 2I_P\not\subseteq \tilde{G}\}
\end{equation*}
satisfies the assumption of Corollary~\ref{cor:impEnergy} with $\lambda=K\abs{F}/\abs{E}$, and hence also the conclusion,
\begin{equation*}
  \operatorname{energy}(\mathbb{P}')\lesssim\frac{\abs{F}}{\abs{E}}.
\end{equation*}

After this, we can proceed exactly as in the proof of Proposition~\ref{prop:largePCarleson} in the case $\abs{E}>\abs{F}$. For $\abs{f}\leq 1_F$ and $\abs{g}\leq 1_{\tilde{E}}\leq 1_E$, we have
\begin{equation*}
\begin{split}
  \sum_{P\in\mathbb{P}'}
  &\abs{\pair{f}{\phi_P}\pair{\phi_P}{ g 1_{\{N(\cdot)\in \omega_u\}}}} \\
  &\lesssim \sum_n\sum_j 
  \operatorname{density}(\mathbb{T}_{n,j})\operatorname{energy}(\mathbb{T}_{n,j})\abs{I_{\mathbb{T}_{n,j}}} \\
  &\lesssim \sum_n\min\{1,\abs{E}2^{-n}\}\min\{\frac{\abs{F}}{\abs{E}},\abs{F}^{\alpha/q}2^{-n\alpha/q}\}2^n,
\end{split}
\end{equation*}
where we substituted the improved energy bound from above in place of the universal energy bound $\operatorname{energy}(\mathbb{P})\lesssim 1$ valid for any collection. The estimation then proceeds with
\begin{equation*}
\begin{split}
  &\leq\sum_{2^n\leq\abs{E}}\frac{\abs{F}}{\abs{E}}2^n+\sum_{\abs{E}\leq 2^n\leq\abs{E}^{q/\alpha}\abs{F}^{1-q/\alpha}}\abs{F} \\
   &\qquad \qquad   +\sum_{2^n\geq\abs{E}^{q/\alpha}\abs{F}^{1-q/\alpha}} \abs{E}\abs{F}^{\alpha/q}2^{-n\alpha/q} \\
  &\lesssim \abs{F}\Big(1+\log\frac{\abs{E}}{\abs{F}}\Big),
\end{split}
\end{equation*}
as we wanted to prove.
\end{proof}

As explained in the beginning of the section, the above two lemmas together with well-known interpolation results complete the proof of Carleson's theorem for $L^p(\R;X)$ with any $p\in(1,\infty)$.

\bibliography{carleson}
\bibliographystyle{plain}

\end{document}